\documentclass{article}[10pt]
\usepackage{float}
\usepackage{amsfonts,amsmath,amsthm}
\usepackage{graphicx}
\usepackage[titletoc,title]{appendix}
\usepackage{multicol}
\usepackage{fullpage}
\usepackage{hyperref}
\hypersetup{
  colorlinks   = true, 
  urlcolor     = blue, 
  linkcolor    = blue, 
  citecolor    = blue  
}

\usepackage{dsfont}
\newtheorem{theorem}{Theorem}[section]
\newtheorem{proposition}[theorem]{Proposition}

\newtheorem{corollary}[theorem]{Corollary}
\newtheorem{lemma}[theorem]{Lemma}
\newtheorem{definition}[theorem]{Definition}

\newcommand{\R}{\mathbb R}
\newcommand{\N}{\mathbb N}
\newcommand{\M}{\mathcal M}
\renewcommand{\P}{\mathcal P}
\newcommand{\1}{\mathbf1}
\newcommand{\A}{\mathcal A}
\newcommand{\B}{\mathcal B}

\def\e{\mathrm{e}}

\DeclareMathOperator{\supp}{supp}

\DeclareFontFamily{U}{matha}{\hyphenchar\font45}
\DeclareFontShape{U}{matha}{m}{n}{
      <5> <6> <7> <8> <9> <10> gen * matha
      <10.95> matha10 <12> <14.4> <17.28> <20.74> <24.88> matha12
      }{}
\DeclareSymbolFont{matha}{U}{matha}{m}{n}
\DeclareFontSubstitution{U}{matha}{m}{n}

\DeclareFontFamily{U}{mathx}{\hyphenchar\font45}
\DeclareFontShape{U}{mathx}{m}{n}{
      <5> <6> <7> <8> <9> <10>
      <10.95> <12> <14.4> <17.28> <20.74> <24.88>
      mathx10
      }{}
\DeclareSymbolFont{mathx}{U}{mathx}{m}{n}
\DeclareFontSubstitution{U}{mathx}{m}{n}

\DeclareMathDelimiter{\vvvert}{0}{matha}{"7E}{mathx}{"17}

\title{The mean-field equation of a leaky integrate-and-fire neural network: measure solutions and steady states}
\date{}
\author{Gr\'egory Dumont \thanks{Group for Neural Theory, LNC INSERM U960, DEC,
Ecole Normale Superieure PSL* University, Paris, France.
Email: gregory.dumont@ens.fr}
\and
Pierre Gabriel \thanks{Laboratoire de Math\'ematiques de Versailles, UVSQ, CNRS, Universit\'e Paris-Saclay,  45 Avenue des \'Etats-Unis, 78035 Versailles cedex, France. Email: pierre.gabriel@uvsq.fr}}

\begin{document}

\maketitle

\abstract{ 
Neural network dynamics emerge from the interaction of spiking cells. 
One way to formulate the problem is through a theoretical framework inspired by ideas coming from statistical physics, the so-called mean-field theory.
In this document, we investigate different issues related to the mean-field description of an excitatory network made up of leaky integrate-and-fire neurons. 
The description is written in the form a nonlinear partial differential equation which is known to blow up in finite time when the network is strongly connected.
We prove that in a moderate coupling regime the equation is globally well-posed in the space of measures, and that there exist stationary solutions.
In the case of weak connectivity we also demonstrate the uniqueness of the steady state and its global exponential stability.
The method to show those mathematical results relies on a contraction argument of Doeblin's type in the linear case, which corresponds to a population of non-interacting units.}

\


\



\section{Introduction}

The dynamics of neural networks is extremely complex. In the brain, a population of neurons is ruled by the interaction of thousands of nervous cells that exchange information by sending and receiving action potentials. Neuroscience needs a theory to relate key biological properties of neurons, with emerging behavior at
the network scale.
From a mathematical perspective, a neural network can simply be seen as a high-dimensional dynamical system of interacting elements. 
Unfortunately, introducing these
interactions tends to lead to models that are analytically
intractable.
Over the past few decades, a big challenge has been to reduce the description of neural circuits.

Most attempts to establish a mathematically tractable characterization of neural networks have made
use of mean-field theory (MT), see \cite{deco2008,Bres01} for a bio-physical review on the subject. Because each neuron receives input from
many others, a single cell is mostly responsive to the average
activity of the population - the mean-field - rather than the specific pattern of individual units. Based on theoretical concepts coming from statistical physics, MT gives rise to a so-called mean-field equation that defines the dynamic of a large (theoretically infinite) population of neurons \cite{deco2008,gerstner}. The use of MT is nowadays well accepted in neuroscience, and
it has already brought significant insights into the emergent properties of neural circuits. For instance, it has played a crucial part in the understanding of neural synchronization and emerging brain rhythms \cite{B02}.

Although MT is widespread among theoreticians, most of the mean-field equations are written within the language of partial differential equations (PDEs) for which there are only few mathematical studies. In this paper, our goal is precisely to fill this gap by considering a mean-field model that prevails in neuroscience. We focus our investigation on the existence and properties of the steady state measure of a PDE that arises for the description of an excitatory network of leaky integrate-and-fire (LIF) neurons.

The LIF model is a well-established neuron model within the neuroscience community \cite{Izi}. It consists of an ordinary differential equation that describes the subthreshold dynamics of a neuron membrane's potential. The equation is endowed with a discontinuous reset mechanism to account for the onset of an action potential. Whenever the membrane potential reaches the firing threshold, the neuron initiates an action potential and the membrane potential is reset, see \cite{Burkitt} for a review and \cite{abbott02,B01} for historical consideration. In its normalized form, the LIF model reads 
\begin{equation*}
\left\{
\begin{array}{l}%
\frac{d}{dt}v(t)=-v(t)+h\sum_{j=1}^{+\infty}\delta (t-t_{j})
\vspace{1mm}\\
\text{If} \quad v>1 \quad \text{then} \quad v\to v_r.
\end{array}%
\right.
\end{equation*}%
Here, $v_r\in(0,1)$ is the reset potential, $\delta$ is the Dirac measure, $h\in(0,1)$ is the so-called synaptic strength, and $t_{j}$ are the arrival times of action potentials that originate from presynaptic cells.
The fact that $h$ is positive means that we consider excitatory neurons.

Due to the presence of Dirac masses, the LIF equation describes a stochastic jump process, or piecewise deterministic Markov process \cite{Davis1984}. Those voltage jumps result from the activation of the synapse at the reception of an action potential, the so-called excitatory postsynaptic potential (EPSP). Note that the stochastic feature of the neural model is embedded in the Poisson distribution of time arrivals \cite{andre01}. It is worth saying that, despite its vast simplifications, the LIF model yields amazingly accurate predictions and is known to reproduce many aspects of actual neural data \cite{naud}. Of course, there have been several variants and generalizations of the model \cite{Izi}.
In Fig.~\ref{Fig01}, a simulation of the LIF model is presented. It illustrates the different processes involved in the membrane equation such as the voltage jumps at the reception of an action potential (i.e. the EPSPs), and the reset mechanism at the initiation of an action potential.

\begin{figure}[]
   \begin{center}
      \includegraphics[width=16.cm]{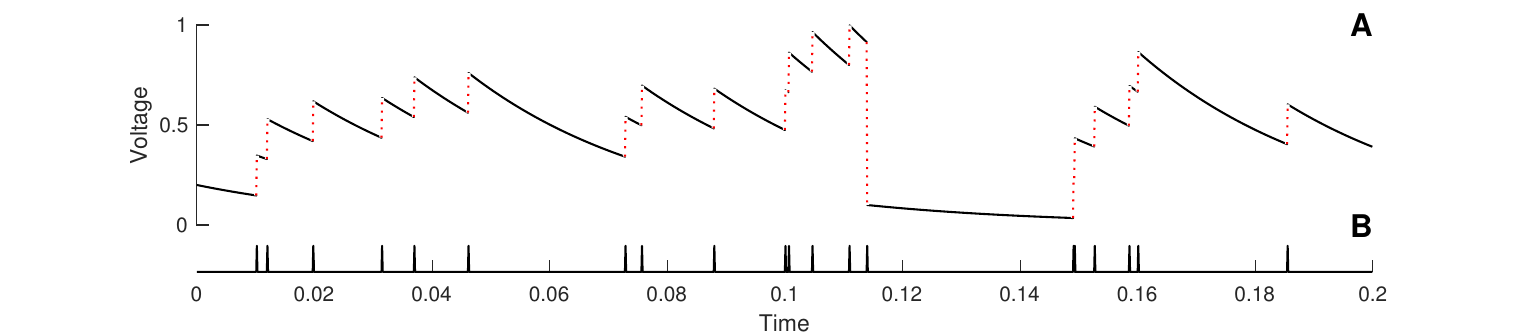}
		   \caption{Simulation of the LIF model. A) Time evolution of the membrane potential. B) The panel illustrates the arrival times of impulses, so-called Poisson spike train. The red dots correspond to discontinuities induced by the jump process. The parameters are: $h=0.2$, $v_r=0.1$ and Poisson rate $100$.}
		\label{Fig01}
		   \end{center}
\end{figure}

In a network, when a cell fires, the dynamics of each other neuron might be affected by the action potential. However, since synaptic transmissions are highly stochastic~\cite{Dobrunz,MaassZador}, the reception of an EPSP only occurs according to a certain probability. This probability plays the role of a coupling parameter.
More precisely it is reflected by a positive number $J$ representing the average number of cells undergoing an EPSP, so that the probability of reception of each neuron in a network of $N$ units is $J/N$.
The dynamics of a neural network made up of LIF neurons is plotted in Fig.~\ref{Fig02}. For each simulation, we show the network raster plot where dots indicate the spiking time of individual units.  The panels correspond to different values of the coupling $J$. As we can see, for weak coupling, the network displays an asynchronous activity where each neuron fires irregularly (Fig.~\ref{Fig02}A). In contrast, when the coupling parameter is taken sufficiently large, the network enters into a synchronous state (Fig.~\ref{Fig02}B). The system seems to have a critical coupling value for which, above this value, the system is driven to a synchronous state, while below this value, it remains asynchronous \cite{cascade01,cascade02}.  A great deal of attention has been devoted to the precise functional forms of these patterns, 
and insight can be gained using MT.

 \begin{figure}[]
   \begin{center}
      \includegraphics[width=16.cm]{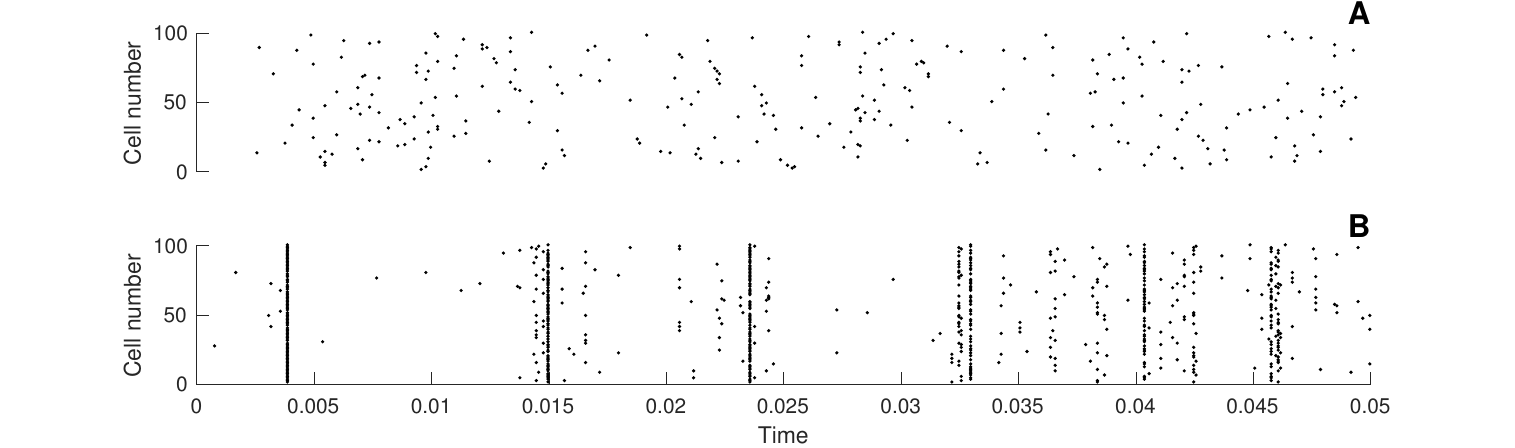}
		   \caption{Simulations of the neural network. The network contains $N = 100$ neurons. In each panel is shown the spiking activity of every
neuron in a raster plot (dots represent spikes). The parameters are: $h=0.1$, $v_r=0.1$ and Poisson rate $200$. The average affected cells $J$ is: A) $J=1$, B) $J=9$. }
		\label{Fig02}
		   \end{center}
\end{figure}

As mentioned above, MT is used to simplify the description of networks and is formally \cite{deco2008,Bres01} or rigorously \cite{DMGLP,FournierLocherbach} derived in the limit of an infinitely large number of elements. In this setting, trajectories of individual units are ignored, and instead, the focus is made on the probability of finding a randomly chosen neuron in a certain state. A continuity equation describing the dynamics of the probability density function (PDF) is then derived, 
and the study of the PDF forms the basis of the mean-field approach. The fundamental assumption at
the core of this theoretical setting is that all the neurons of the network share similar biophysical properties. 

A pioneering attempt to describe neural networks within the framework of MT was made around the 1970s with the seminal work of Wilson and Cowan, followed by a paper of Amari \cite{WC00,A01}. Since then the study of neural circuits within the mean-field approach
 has never lost interest within the scientific community. To mention just a few,  
Sirovich, Omurtag and Knight \cite{sirovich02},  Nykamp and Tranchina \cite{Nykamp2000}, Brunel and Hakim \cite{BH01,B02}, and the work of Gerstner \cite{gerstner2000population}, were among the first to study networks of LIF neurons using MT.

In the present work, we are interested in the mean-field equation of a LIF neural network derived in~\cite{sirovich}, see also~\cite{Nykamp2000}.
We denote the probability density function $p(t,v)$, such that $Np(t, v)dv$ gives the approximate number of neurons with membrane potential in $[v-dv,v)$ at time $t$ for a network made up of $N$ neurons. 
It is assumed that each neuron receives excitatory synaptic input with average rate  $\sigma (t)$ and fires action potentials at rate $r(t)$.
Then the dynamics of the density $p(t,v)$ is prescribed by the following nonlinear partial differential equation:
\begin{equation}\label{eq:LIF}
\frac{\partial }{\partial t}p(t,v)-\underbrace{\frac{\partial }{\partial v}[vp(t,v)]}_{\text{Leak}}
+\underbrace{\sigma (t) \big[p(t,v)  - p(t,v-h)\1_{[h,1)}(v)\big] }_{\text{Jump}}  =
\underbrace{\delta (v-v_{r}) r(t)}_{\text{Reset}},\qquad0<v<1,
\end{equation}
complemented with a zero flux boundary condition
\[p(t,1)=0.\]
Note that a similar equation is introduced in the influential textbook \cite{gerstner2014neuronal}.
We show in Fig.~\ref{State_Potential} a schematic representation of the state space for the mean-field equation where the different operators take place. The jump process of the mean-field equation accounts for stochastic EPSPs arrival at the cellular level.

\begin{figure}[]
\begin{center}
    \includegraphics[width=6cm]{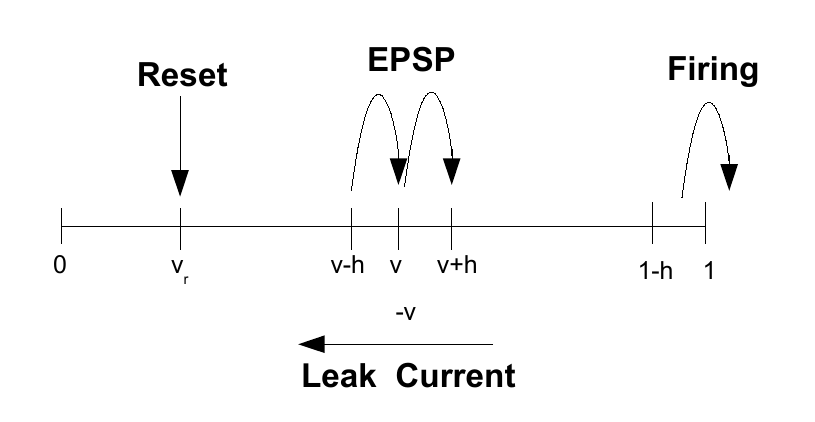}
  \caption{Schematic representation of the state space for the mean-field equation. }\label{State_Potential}
      \end{center}
\end{figure}

The firing activity of the network $r(t)$ is easily extracted from the mean-field equation. The proportion of cells crossing the threshold is given by the quantity of cells with potential between $1-h$ and $1$ that undergo a jump, see Fig.~\ref{State_Potential}, and so
\[r(t)=\sigma (t)\int_{1-h}^1 p(t,w) \,dw\]
since $\sigma(t)$ is the jump rate.
This expression guarantees that Equation~\eqref{eq:LIF} is formally conservative, in the sense that the integral of the solution is preserved along time. This is required since it should describe the evolution of a probability density.
If the arrival rate $\sigma(t)$ is a positive constant $\sigma_0$,
Equation~\eqref{eq:LIF} is the Kolmogorov forward equation of the LIF stochastic differential equation where the time arrivals of action potentials are distributed according to a Poisson process of rate $\sigma_0$.
In other words the probability that the potential $v(t)$ of a single LIF neuron (or each neuron of an unconnected population) belongs to a measurable subset $A$ of $(0,1)$ is given by $\int_Ap(t,w)dw$.
In the case of a connected network, the external arrival rate is modulated by the reception of EPSP emitted by the other neurons.
In the mean-field framework, it is assumed that single neurons are only sensitive to the average population activity \cite{sirovich}.
This postulate leads to the following expression of the arrival rate $\sigma (t)$ as the sum of an external rate $\sigma_0$ and the firing rate multiplied by the average number $J$ of synaptic connections
\[\sigma (t)=\sigma_0 +J r(t).\]
Combining the two above relations between $\sigma(t)$ and $r(t)$ we get an explicit formula for the arrival rate
\[\sigma(t)=\frac{\sigma_0}{1-J\int_{1-h}^1p(t,w)dw},\]
provided that the denominator is positive.
This makes the value $J=1$ appear critical and suggests that for $J>1$ some blow-up phenomena should occur when the initial distribution is concentrated enough around $v=1$, see~\cite{carrillo01} for a PDE view on this phenomenon and~\cite{Delarue1,Delarue2} for a stochastic perspective.
Actually it has been shown that the solutions to Equation~\eqref{eq:LIF} blow-up in finite time for any initial data in the strong connectivity regime \cite{DH02}. This was
attributed to the instantaneity of spikes firings and their immediate effects on the
firing of other cells. This is happening when 
\begin{equation*}
J\geq 1+\frac{1-v_r}{h} \qquad\text{and}\qquad h\sigma_0>1.
\end{equation*}
When $J<1$ the arrival rate is always well defined, no blow-up can occur, and the solutions exist for all time whatever the initial data \cite{DH}. Figure \ref{Fig03} portrays the dynamics of the mean-field equation in such a situation, with a Gaussian profile as initial condition (Fig.~\ref{Fig03}A). Under the drift and the jump process, the density function gives a non zero flux at the threshold, and this flux is reinjected right away according to the reset process. This effect can be clearly seen in the third panel of the simulation presented in Fig.~\ref{Fig03}B. Asymptotically, the solution reaches a stationary profile which is shown in Fig.~\ref{Fig03}C.

Although progresses have been made, several questions remain unanswered, specially in the moderate or weak connectivity regime. For instance, as we can see from some simulation presented above, we observe that the density converges toward a stationary state. Can we show the existence of a steady state? Can we analyze its stability properties? Answering these questions will allow us to form a deeper understanding of the asynchronous states of neural networks. Our challenge is to study the existence and properties of the mean-field equation steady states.

\

The paper is structured as follows. In Section~\ref{sec:main} we introduce the main notations and definitions, and we give a summary of the main results obtained throughout this manuscript.
Section~\ref{sec:linear_case} is devoted to the study of the linear regime, which corresponds to a population of uncoupled neurons ($J=0$).
More precisely we prove the well-posedness of the equation in the space of measures and, via a so-called Doeblin condition, the exponential convergence to an asynchronous state.
This is a crucial preliminary step before studying the nonlinear case.
In Section~\ref{sec:globalwp} we prove the existence and uniqueness of global in time measure solutions to Equation~\eqref{eq:LIF} in the moderate nonlinear regime $J<1.$
Section~\ref{sec:steadystates} deals with the stationary solutions and their possible exponential stability.
We show the existence of at least one steady state when
\[J<1+\left\lfloor\frac{1-v_r}{h}\right\rfloor\]
and the existence of at least two steady states when
\[J>1+\left\lfloor\frac{1-v_r}{h}\right\rfloor\qquad\text{and}\qquad\sigma_0<\frac{1-h}{4J}.\]
Eventually, we demonstrate the global exponential stability of the (unique) steady state in the weakly nonlinear regime $J\ll1.$
This work complements results on asynchronous state in different models~\cite{CanizoYoldas,MQT16,PakdamanPerthameSalort14,PerthameSalort}.

  \begin{figure}[]
   \begin{center}
      \includegraphics[width=16.cm]{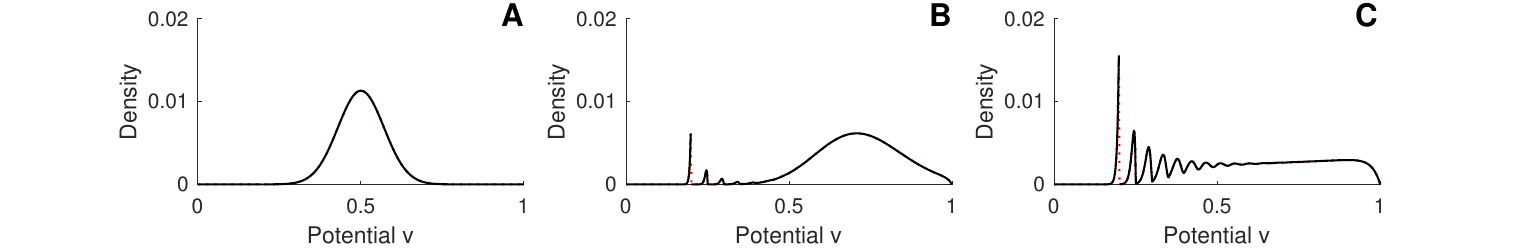}
		   \caption{Simulations of the MF equation. A gaussian
was taken as initial condition. The plots show in blue the evolution in time of the solution at different times. The red dots correspond to the discontinuity induced by the reset process. The parameters of the simulation are: $v_r = 0.3$, $h=0.05$, $\sigma_0 =50$, $J=0$ A) $t = 0$, B) $t = 0.12$, C) $t = 7$. }
		\label{Fig03}
		   \end{center}
\end{figure}

\section{Main results}\label{sec:main}

Measure solutions to structured population equations has attracted increasing interest over the past few years~\cite{CanizoCarrilloCuadrado,CarrilloColomboGwiazdaUlikowska,PG,GwiazdaLorenzMarciniak,GwiazdaWiedemann}.
In this paper we are concerned with measure solutions to the mean-field description of a LIF neural network given by Equation~\eqref{eq:LIF}.
Measure theory offers a very natural framework for two reasons.
First it allows to consider a Dirac mass initial distribution. Such an initial profile corresponds to a fully synchronous state and is thus perfectly relevant in neuroscience.
Second it is very well suited for dealing with equations having a singular source term (the reset part in Equation~\eqref{eq:LIF}).

\medskip

Before giving the definition of such solutions, we recall some results about measure theory (and we refer to~\cite{Rudin} for more details).
We endow the interval $[0,1]$ with its usual topology and the associated Borel $\sigma$-algebra.
We denote by $\M([0,1])$ the space of signed Borel measures on $[0,1],$
 by $\M_+([0,1])$ its positive cone (the set of finite positive Borel measures),
and by $\P([0,1])$ the set of probability measures.
The Jordan decomposition theorem ensures that any $\mu$ belonging to $\M([0,1])$ admits a unique decomposition
 $$\mu=\mu_+-\mu_-,$$
where $\mu_+$ and $\mu_-$ are positive and mutually singular.
The space $\M([0,1])$ is endowed with the total variation norm defined for all $\mu$ belonging to $\M([0,1])$ by
\[\left\|\mu\right\|_{\mathrm{TV}}:=\mu_+([0,1])+\mu_-([0,1]).\]
For any bounded Borel function $f$ on $[0,1]$ the supremum norm is defined by $$\left\|f\right\|_\infty=\sup_{0\leq v\leq 1}|f(v)|$$
and for any $\mu$ belonging to $\M([0,1])$ we use the notation
\[\mu f:=\int_{[0,1]}f\,d\mu.\]
Endowed with the supremum norm, the space $C([0,1])$ of continuous functions on $[0,1]$ is a Banach space.
The Riesz representation theorem ensures that $\M([0,1])$ can be identified with the topological dual space of $C([0,1])$ through the mapping
\[\begin{array}{ccc}
\M([0,1])&\to&C([0,1])'
\vspace{1mm}\\
\mu&\mapsto&\{f\mapsto\mu f\}
\end{array}\]
which is an isometric isomorphism:
\[\left\|\mu\right\|_{\mathrm{TV}}=\sup_{\|f\|_\infty\leq1}|\mu f|.\]
Recall that a sequence $(\mu_n)_{n\in\N}$ taken from $\M([0,1])$ is said to converge weak* to $\mu$ an element of $\M([0,1])$ if $(\mu_n f)_{n\in\N}$ converges to $\mu f$ for all $f$ belonging to $ C([0,1]).$

\

Now we can give the definition of a measure solution to Equation~\eqref{eq:LIF}.
We use the notation $\1_\Omega$ for the indicator function of a subset $\Omega$ of $[0,1],$ and we simply denote by $\1$ the constant function $\1_{[0,1]}.$

\begin{definition}\label{def:sol}
Let $T>0$, we say that a family $(\mu_t)_{t\geq0}$ of $\P([0,1])$ is a solution to Equation~\eqref{eq:LIF} on $[0,T)$ with initial datum $\mu_0$ if
\begin{itemize}
\item $t\mapsto\sigma(t):=\dfrac{\sigma_0}{1-J\mu_t([1-h,1])}$ is positive and locally integrable on $[0,T),$
\item $t\mapsto\mu_t$ is weak*-continuous on $[0,T),$
\item  and for all $f\in C^1([0,1])$ and all $t\in[0,T)$
\[\mu_tf=\mu_0f+\int_0^t\int_{[0,1]}\Big(\!-vf'(v)+\sigma(s)\big[f(v+h)\1_{[0,1-h)}(v)+f(v_r)\1_{[1-h,1]}(v)-f(v)\big]\Big)d\mu_s(v)\,ds.\]
\end{itemize}
\end{definition}

For the sake of simplicity, it is useful to define the following operators.
For any $f$ taken from $ C([0,1])$ we set
\[\B f(v)=f(v+h)\1_{[0,1-h)}(v)+f(v_r)\1_{[1-h,1]}(v)-f(v),\]
and, for any $f$ belonging to $C^1([0,1])$ and $\sigma>0,$
\[\A_\sigma f(v)=-vf'(v)+\sigma\,\B f(v).\]
With this definition the weak formulation of Equation~\eqref{eq:LIF} in Definition~\ref{def:sol} reads
\[\mu_tf=\mu_0f+\int_0^t\mu_s(\A_{\sigma(s)}f)\,ds.\]
Notice that $\A$ and $\B$ are conservative in the sense that $$\A\1=\B\1=0.$$
Notice also that $\B$ is a bounded operator in the sense that
\[\forall f\in C([0,1]),\qquad\|\B f\|_\infty\leq2\,\|f\|_\infty,\]
but in general $\B f $ is not a continuous function, and thus neither $\A_\sigma f$ when $f$ is taken from $ C^1([0,1]).$
This prevents the use of general results about the existence and uniqueness of measure solutions for structured population models
(see~\cite{CanizoCarrilloCuadrado,CarrilloColomboGwiazdaUlikowska,GwiazdaLorenzMarciniak}).
As we shall see, to prove the well-posedness of our problem, we use a duality method that is well suited for analysing steady states.

\

We can now present the main results of the paper regarding the mean-field description of LIF neural networks given by Equation~\eqref{eq:LIF}. Before that, let us mention that to avoid pathological situations where, starting from the reset potential $v_r,$ the potential can reach exactly the threshold $1$ by doing only jumps, we always assume that 
\[\frac{1-v_r}{h}\not\in\N.\]
The main results can be summarized by the two following theorems.

\medskip

\begin{theorem}\label{th:globalwp}
Assume that $J<1.$ Then for any initial probability measure $\mu_0$ there exists a unique global measure solution $(\mu_t)_{t\geq0}$ to Equation~\eqref{eq:LIF}, in the sense of Definition~\ref{def:sol}.
\end{theorem}

\medskip

Let us remind the reader that for a connectivity $J<1,$ we already knew from~\cite{DH} that the mean-field description given by Equation~\eqref{eq:LIF} is globally well-posed in $L^1([0,1]).$
Theorem~\ref{th:globalwp} ensures that it is still the case in the larger space $\M([0,1]).$ The second theorem is about the steady states, {\it i.e.}
probability measures $\bar\mu$ which satisfy
\[\forall f\in C^1([0,1]),\quad \bar\mu(\A_{\bar\sigma}f)=0,\qquad\text{where}\quad\bar\sigma=\dfrac{\sigma_0}{1-J\bar\mu([1-h,1])}.\]
We already know from~\cite{DH02} that under the conditions
\begin{equation*}
J\geq 1+\frac{1-v_r}{h} \qquad\text{and}\qquad h\sigma_0>1
\end{equation*}
no steady state can exist since there is blow-up in finite time whatever the initial distribution.
The following theorem provides sufficient conditions for existence that can be compared to the above non-existence conditions, and a uniqueness and stability result in the case of small connectivity.

\begin{theorem}\label{th:steadystate}
Depending on the network connectivity, the following situations occur:
\begin{enumerate}
\item[(i)] Under the conditions
\[J>1+\left\lfloor\frac{1-v_r}{h}\right\rfloor\qquad\text{and}\qquad\sigma_0<\frac{1-h}{4J},\]
there exist at least two steady states.
\item[(ii)] If the following inequality holds
\[J<1+\left\lfloor\frac{1-v_r}{h}\right\rfloor,\]
then there exists at least one steady state.
\item[(iii)] In the case when
\[J<(5-2\sqrt6)\Big(\frac h4\Big)^{\sigma_0+1},\]
the steady state $\bar\mu$ is unique and globally exponentially stable.
More precisely there exist explicit constants $t_0,a>0$ such that for all $\mu_0$ taken from $\P([0,1])$ and all $t\geq0$
\[\big\|\mu_t-\bar\mu\big\|_{\mathrm{TV}}\leq \e^{-a(t-t_0)}\big\|\mu_0-\bar\mu\big\|_{\mathrm{TV}},\]
where $(\mu_t)_{t\geq0}$ is the unique solution to Equation~\eqref{eq:LIF} with initial datum $\mu_0.$
\end{enumerate}
\end{theorem}

To our knowledge, it is the first time such a result is proved for the mean-field LIF model.
Before this work, the steady state analysis was only performed on the diffusion approximation of Equation~\eqref{eq:LIF}
in~\cite{carrillo01,CarrilloPerthameSalortSmets}, see the conclusion section~\ref{sec:conclusion} for a more detailed discussion.

\section{The linear case}\label{sec:linear_case}

Our strategy to prove Theorem \ref{th:globalwp} and Theorem \ref{th:steadystate} relies on a careful study of the mean-field equation when $J$ is taken to be zero.
This particular case corresponds to the description of an unconnected population of neurons. In this setting, the equation becomes linear and reads
\begin{equation}\label{eq:linear}
\frac{\partial }{\partial t}p(t,v)-\frac{\partial }{\partial v}\big[vp(t,v)\big]+ \sigma_0 \big[p(t,v)  - p(t,v-h)\1_{[h,1)}(v)\big]=\bigg[\sigma_0 \int_{[1-h,1]} p(t,w) \,dw\bigg]\delta_{v=v_r}.
\end{equation}
Notice that this equation without the reset part, {\it i.e.} without the Dirac mass source term, was studied in~\cite{Heijmans}.
For the sake of clarity in the current section we will denote by $\A$ the operator $\A_{\sigma_0},$ $\sigma_0$ being a fixed positive number.
Since the equation is linear, we do not need to restrict the definition of a solution to probability measures.
We say that a family of measures $(\mu_t)_{t\geq0}$ is a solution to Equation~\eqref{eq:linear} with initial datum $\mu_0$ when the mapping $t\mapsto\mu_t$ is weak*-continuous and for all $f $ in $ C^1([0,1])$ and all $t\geq0$
\[\mu_tf=\mu_0f+\int_0^t\mu_s\A f\,ds.\]

\begin{theorem}\label{th:gen_Mt}
Equation~\eqref{eq:linear} generates a weak*-continuous semigroup on $\M([0,1]),$
{\it i.e.} there exists a semigroup $(M_t)_{t\geq0}$ of linear operators, which map any signed measure $\mu$ to another one $\mu M_t$ for any $t$,
such that for any initial measure $\mu_0$ the unique solution to Equation~\eqref{eq:linear} is given by the family $(\mu_0 M_t)_{t\geq0}.$
Additionally the set of probability measures $\P([0,1])$ is invariant under $(M_t)_{t\geq0}.$
In particular $(M_t)_{t\geq0}$ is a positive contraction semigroup.
\end{theorem}

\begin{proof}

In order to build the semigroup $(M_t)_{t\geq0}$ we follow the method in~\cite{PG}, which is based on the dual equation
\[\partial_tf(t,v)+v\partial_vf(t,v)+\sigma_0f(t,v)=\sigma_0\big[f(t,v+h)\1_{[0,1-h)}(v)+f(t,v_r)\1_{[1-h,1]}(v)\big],\]
with the initial condition $f_0.$
This equation is well-posed in the space of continuous functions, in the sense of the following lemma.

\begin{lemma}\label{lm:fixedpoint}
For any $f_0$ belonging to $C([0,1]),$ there exists a unique $f$ in $C(\R_+\times[0,1])$ which satisfies
\[f(t,v)=f_0(v\e^{-t})\e^{-\sigma_0t}+\sigma_0\int_0^t\e^{-\sigma_0\tau}\big[f(t-\tau,\e^{-\tau}v+h)\1_{[0,1-h)}(\e^{-\tau}v)+f(t-\tau,v_r)\1_{[1-h,1]}(e^{-\tau}v)\big]d\tau.\]
Additionally:
 $$f_0=\1 \Rightarrow f=\1\qquad\text{and} \qquad f_0\geq0 \Rightarrow f\geq0.$$
\end{lemma}

\begin{proof}
The proof consists in applying the Banach fixed point theorem.
Fix $T>0$ and define on the Banach space $C([0,T]\times[0,1])$ endowed with the supremum norm $\|\cdot\|_\infty$ the mapping $\Gamma$ by
\[\Gamma f(t,v):=f_0(v\e^{-t})\,\e^{-\sigma_0t}+\sigma_0\int_0^t\e^{-\sigma_0\tau}\big[f(t-\tau,v\e^{-\tau}+h)\1_{[0,1-h)}(v\e^{-\tau})+f(t-\tau,v_r)\1_{[1-h,1]}(v\e^{-\tau})\big]\,d\tau.\]
It is a contraction whatever the value of $T.$
Indeed it is an affine mapping and for $f_0\equiv 0$ we have
\[\|\Gamma f\|_\infty\leq\sigma_0\int_0^t\e^{-\sigma_0\tau}\|f\|_\infty\,d\tau= (1-\e^{-\sigma_0T})\|f\|_\infty.\]
The Banach fixed point theorem ensures the existence and uniqueness of a fixed point for $\Gamma$ in $C([0,T]\times[0,1]),$ for any $T>0.$
It is easy to check that $f=\1$ is a fixed point when $f_0=\1,$ and for the positivity it suffices to check that when $f_0\geq0,$ the 
 positive cone of $C([0,T]\times[0,1])$ is invariant under $\Gamma.$
\end{proof}

With this result we can define a family $(M_t)_{t\geq0}$ of linear operators on $C([0,1])$ by setting
\[M_tf_0:=f(t,\cdot).\]
The family $(M_t)_{t\geq0}$ thus defined is a semigroup, meaning that for all $f$ taken from $C([0,1])$ and all $s,t\geq0$
\[M_0f=f\qquad\text{and}\qquad M_{t+s}f=M_t(M_sf).\]
It is a consequence of the uniqueness in Lemma~\ref{lm:fixedpoint}, since 
$$(t,v)\mapsto M_{t+s}f(v)\qquad\text{and}\qquad (t,v)\mapsto M_t(M_sf)(v)$$ 
are both a fixed point of $\Gamma$ for $f_0=M_sf.$
Moreover this semigroup is conservative and positive, in the sense that for all $t\geq0$
\[M_t\1=\1\qquad\text{and}\qquad f\geq0\ \implies\ M_tf\geq0.\]
As a direct consequence it is a contraction for the supremum norm, meaning that for all $f\in C([0,1])$ and all $t\geq0$
\[\|M_tf\|_\infty\leq\|f\|_\infty.\]

Now we define by duality a semigroup on $$\M([0,1])=C([0,1])'.$$
For $\mu$ belonging to $\M([0,1])$ and $t\geq0$ we define $\mu M_t$ an element of $\M([0,1])$ by
\begin{equation}\label{def:left_Mt}
\forall f\in C([0,1]),\qquad (\mu M_t)f:=\mu(M_tf).
\end{equation}
The properties of the right action of $(M_t)_{t\geq0}$ are readily transfered to the left action by duality.
The left semigroup $(M_t)_{t\geq0}$ defined on $\M([0,1])$ by~\eqref{def:left_Mt} is conservative and positive, in the sense that for all $t\geq0$
\[\mu M_t([0,1])=\mu([0,1])\qquad\text{and}\qquad\mu\in\M_+([0,1])\quad\implies \quad \mu M_t\in\M_+([0,1]).\]
As a consequence it leaves invariant $\P([0,1])$ and it is a contraction for the total variation norm: for all signed measure $\mu$ and all $t\geq0$
\[\left\|\mu M_t\right\|_{\mathrm{TV}}=\left\|\mu_+ M_t-\mu_-M_t\right\|_{\mathrm{TV}}\leq
(\mu_+M_t)([0,1])+(\mu_-M_t)([0,1])=\mu_+([0,1])+\mu_-([0,1])=\left\|\mu\right\|_{\mathrm{TV}}.\]
The verification that the family $(\mu M_t)_{t\geq0}$ is the unique solution to Equation~\eqref{eq:linear} with initial datum $\mu$ requires regularizing the equation.
The rather technical details are postponed in Appendix~\ref{app:linear}, and this ends the proof of Theorem~\ref{th:gen_Mt}.
\end{proof}

\

Now we give a crucial ergodic result about the semigroup $(M_t)_{t\geq0}.$

\begin{theorem}\label{th:main_lin}
The semigroup $(M_t)_{t\geq0}$ admits a unique invariant probability measure $\bar\mu,$ {\it i.e.} there exists a unique $\bar\mu$ element of $\P([0,1])$ such that for all positive time $t$
\[\bar\mu M_t=\bar\mu.\]
This invariant measure is globally exponentially stable: for all $\mu$ belonging to $\M([0,1])$ and for all  $t$ being positive
\[\big\|\mu M_t-(\mu\1)\bar\mu\big\|_{\mathrm{TV}}\leq \e^{-a(t-t_0)}\big\|\mu-(\mu\1)\bar\mu\big\|_{\mathrm{TV}},\]
where the constants $t_0$ and $a$ are given by:
 $$t_0=\log\frac4h>0\qquad\text{and} \qquad a=\frac{-\log\big(1-\frac{\sigma_0}{2}(\frac h4)^{\!^{\sigma_0}}\big)}{\log\frac4h}>0.$$
\end{theorem}

Notice that the values of $t_0$ and $a$ are explicit (in terms of the coefficients of the model) but not optimal.
The optimization of these constants is also an interesting issue that could be addressed in a future work.

The proof of Theorem~\ref{th:main_lin} relies on a contraction property obtained via a so-called Doeblin condition, see~\cite{Bansaye,PG} for recent presentations and developments on this method, or~\cite{CanizoYoldas} for an application to elapsed-time neural models, and also~\cite{PichorRudnicki} where a similar type of condition is used for the study of a Stein neural model related to ours.
More precisely we use the following well-known result, of which we give a short proof for the sake of completeness.

\begin{proposition}\label{prop:expo_contraction}
Let $(M_t)_{t\geq0}$ be a semigroup which leaves invariant $\P([0,1])$ and satisfies the Doeblin condition
\[\exists\, t_0>0, c\in(0,1), \nu\in\P([0,1])\quad\text{such that}\quad\forall\mu\in\P([0,1]),\ \mu M_{t_0}\geq c\,\nu.\]
Then for all $\mu,\widetilde\mu$ elements of $\P([0,1])$ we have
\[\forall t\geq0,\qquad\left\|\mu M_t-\widetilde\mu M_t\right\|_{\mathrm{TV}}\leq\e^{-a(t-t_0)}\left\|\mu-\widetilde\mu\right\|_{\mathrm{TV}}\]
with
\[a=\frac{-\log(1-c)}{t_0}>0.\]
\end{proposition}

\begin{proof}
Let $\mu$ and $\widetilde\mu$ be two probability measures on $[0,1]$
and define
\[\bar\mu:=\frac2{\left\|\mu-\widetilde\mu\right\|_{\mathrm{TV}}}(\mu-\widetilde\mu).\]
Since
\[(\mu-\widetilde\mu)_+([0,1])=(\mu-\widetilde\mu)_-([0,1])=\frac12\left\|\mu-\widetilde\mu\right\|_{\mathrm{TV}},\]
the positive part $\bar\mu_+$ and the negative part $\bar\mu_-$ of $\bar\mu$ are probability measures.
By virtue of Doeblin's condition we have
\[\bar\mu_\pm M_{t_0}\geq c\nu\]
and we deduce that
\[\left\|\bar\mu_\pm M_{t_0}-c\nu\right\|_{\mathrm{TV}}=(\bar\mu_\pm M_{t_0}-c\nu)([0,1])=1-c.\]
This property leads to
\[\left\|\bar\mu M_{t_0}\right\|_{\mathrm{TV}}\leq\left\|\bar\mu_+M_{t_0}-c\nu\right\|_{\mathrm{TV}}+\left\|\bar\mu_-M_{t_0}-c\nu\right\|_{\mathrm{TV}}=2(1-c),\]
and then
\[\left\|\mu M_{t_0}-\widetilde\mu M_{t_0}\right\|_{\mathrm{TV}}=\frac12\left\|\mu-\widetilde\mu\right\|_{\mathrm{TV}}\left\|\bar\mu M_{t_0}\right\|_{\mathrm{TV}}\leq(1-c)\left\|\mu-\widetilde\mu\right\|_{\mathrm{TV}}.\]
Now for $t\geq0$ we define $n=\big\lfloor\frac{t}{t_0}\big\rfloor$ and we get by induction
\[\left\|\mu M_t-\widetilde\mu M_t\right\|_{\mathrm{TV}}\leq(1-c)^n\left\|\mu M_{t-nt_0}-\widetilde\mu M_{t-nt_0}\right\|_{\mathrm{TV}}\leq \e^{n\log(1-c)}\left\|\mu-\widetilde\mu\right\|_{\mathrm{TV}}.\]
This ends the proof since
\[n\log(1-c)\leq\Big(\frac{t}{t_0}-1\Big)\log(1-c)=-a(t-t_0).\]

\end{proof}

\begin{proof}[Proof of Theorem~\ref{th:main_lin}]
The first step consists in proving that the semigroup $(M_t)_{t\geq0}$ satisfies the Doeblin condition
\[\forall f\geq0,\ \forall v\in[0,1],\qquad M_{t_0}f(v)\geq c\,(\nu f),\]
with $\nu=\frac{2}{h}\1_{[\frac h2,h]}$ the uniform probability measure on $[\frac h2,h]$ and the following constants
$$t_0=\log\frac 4h>0, \quad c=\frac{\sigma_0}{2}\big(\frac{h}{4}\big)^{\sigma_0}\in(0,1).$$
We start with the definition of $(M_t)_{t\geq0}$ which gives for $f\geq0$
\begin{align*}
M_tf(v)&=f(v\e^{-t})\e^{-\sigma_0t}+\sigma_0\int_0^t\e^{-\sigma_0\tau}\big[M_{t-\tau}f(\e^{-\tau}v+h)\1_{[0,1-h)}(\e^{-\tau}v)+M_{t-\tau}f(v_r)\1_{[1-h,1]}(\e^{-\tau}v)\big]d\tau\\
&\geq f(v\e^{-t})\e^{-\sigma_0t}+\sigma_0\int_0^t\e^{-\sigma_0\tau}M_{t-\tau}f(\e^{-\tau}v+h)\1_{[0,1-h)}(\e^{-\tau}v)\,d\tau.
\end{align*}
Iterating this inequality we deduce
\begin{align*}
M_tf(v)&\geq f(v\e^{-t})\e^{-\sigma_0t}+\sigma_0\int_0^t\e^{-\sigma_0t}f((\e^{-\tau}v+h)\e^{-(t-\tau)})\1_{[0,1-h)}(\e^{-\tau}v)\,d\tau\\
&\geq\sigma_0\e^{-\sigma_0t}\int_0^tf((\e^{-\tau}v+h)\e^{-(t-\tau)})\1_{[0,1-h)}(\e^{-\tau}v)\,d\tau.
\end{align*}
Let $t_1=-\log h$ the time after which all the neurons which did not undergo potential jumps have a voltage between $0$ and $h$
{\it i.e.} 
$$\forall \tau\geq t_1, \quad \forall v\in[0,1], \quad v\e^{-\tau}\in[0,h],$$
and let $t_2>0$ to be chosen later.
For $$t=t_0:=t_1+t_2,$$ we have
\begin{align*}
M_{t_0}f(v)&\geq\sigma_0\e^{-\sigma_0t_0}\int_{t_1}^{t_0}f(\e^{-t_0}v+h\e^{-(t_0-\tau)})d\tau\\
&\geq\sigma_0\e^{-\sigma_0t_0}\int_{t_1}^{t_0}f(\e^{-t_0}v+h\e^{-(t_0-\tau)})\,\e^{-(t_0-\tau)}d\tau\\
&=\frac{\sigma_0}{h}\e^{-\sigma_0t_0}\int_{v\e^{-t_0}+h\e^{-t_2}}^{v\e^{-t_0}+h}f(w)\,dw \hspace{40mm} \big(w=\e^{-t_0}v+h\e^{-(t_0-\tau)}\big)\\
&\geq\frac{\sigma_0}{h}\e^{-\sigma_0t_0}\int_{2h\e^{-t_2}}^{h}f(w)\,dw
\end{align*}
For the last inequality we have used that $$v\e^{-t_0}\leq\e^{-t_1-t_2}= h\e^{-t_2}.$$
So if we choose $t_2=\log4$ we get
\[M_{t_0}f(v)\geq\frac{\sigma_0}{2}\e^{-\sigma_0t_0}\frac{2}{h}\int_{\frac h2}^{h}f(w)\,dw=\frac{\sigma_0}{2}\Big(\frac{h}{4}\Big)^{\sigma_0}\nu(f)\]
and the Doeblin condition is proved.

\

As a consequence, Proposition~\ref{prop:expo_contraction} ensures that the mapping $$\mu\mapsto\mu M_{t_0}$$ is a contraction in the complete metric space $(\P([0,1]),\|\cdot\|_{\mathrm{TV}}),$
which therefore admits a unique fixed point $\bar\mu $ in $\P([0,1]).$
The semigroup property ensures that for all $t\geq0,$ $\bar\mu M_t$ is also a fixed point of $M_{t_0}.$
By uniqueness we get that $\bar\mu M_t=\bar\mu,$ meaning that $\bar\mu$ is invariant under $(M_t)_{t\geq0}.$
This concludes the proof since the exponential convergence is an immediate consequence of Proposition~\ref{prop:expo_contraction}.
\end{proof}

\section{Global well-posedness for $J<1$}\label{sec:globalwp}

The aim of this section is to prove Theorem~\ref{th:globalwp}.
Our method of proof relies on duality arguments and divides into several steps.
First we remark that if $(\mu_t)_{t\geq0}$ is a measure solution to Equation~\eqref{eq:LIF}
and $\psi(s,t,v)$ satisfies the nonlinear and nonlocal equation
\[\partial_s\psi(s,t,v)=v\partial_v\psi(s,t,v)+\frac{\sigma_0}{1-J\mu_0\psi(0,s,\cdot)}\big[\psi(s,t,v)-\psi(s,t,v+h)\1_{[0,1-h)}(v)-\psi(s,t,v_r)\1_{[1-h,1]}(v)\big]\]
with the terminal condition $\psi(t,t,v)=\1_{[1-h,1]}(v)$,
then we have
$$\mu_t([1-h,1])=\mu_0\psi(0,t,\cdot).$$
So if we know such a function $\psi$ we can deduce the value of $\sigma(t)$ and we can see Equation~\eqref{eq:LIF} as a time-inhomogeneous but linear equation.
We solve this equation in a similar way than the linear case.
This method of construction allows us to prove a Duhamel formula for Equation~\eqref{eq:LIF} which is then used to prove uniqueness.
The Duhamel formula is also the corner stone to prove the exponential stability of the steady state in the weakly connected regime (Section~\ref{sec:steadystates}).

The discontinuity of the indicator function $\1_{[1-h,1]}$ implies that $\psi,$
if it exists, is a discontinuous function.
This brings difficulties since the duality approach requires to work in the space of continuous functions.
To work around this problem, we approximate the indicator function $\1_{[1-h,1]}$ by
\[\chi_n(v):=\left\{\begin{array}{ll}
0&\text{if}\ v\leq 1-h-\frac hn,
\vspace{2mm}\\
1+\dfrac nh(v-1+h)\quad&\text{if}\ 1-h-\frac hn\leq v\leq 1-h,
\vspace{2mm}\\
1&\text{if}\ v\geq1-h,
\end{array}\right.\]
where $n\in\N^*,$ before passing to the limit $n$ goes to infinity.
The family $(\chi_n)_{n\geq1}$ is a decreasing sequence of continuous functions which converges pointwise to $\1_{[1-h,1]}.$

\medskip

In what follows, $\mu$ {\bf is a fixed probability measure}.
The first step consists in building, for $t\geq0,$ a regularized version of the function $\psi.$
For $T>0$ we denote by $X^T$ the space of three variables continuous functions on the set
\[\{(s,t,v),\,0\leq s\leq t<T,\,0\leq v\leq1\}\]
and for $n\in\N^*$  and $T$ small enough we define $\psi_n\in X^T$ as the unique solution to the nonlinear equation
\[\partial_s\psi_n(s,t,v)=v\partial_v\psi_n(s,t,v)+\frac{\sigma_0}{1-J\mu\psi_n(0,s,\cdot)}\big[\psi_n(s,t,v)-\psi_n(s,t,v+h)(1-\chi_n(v))-\psi_n(s,t,v_r)\chi_n(v)\big],\]
with the terminal condition $$\psi_n(t,t,v)=\chi_n(v).$$
More precisely $\psi_n$ is defined in the following lemma,
where we have set
\[T^*:=\frac{(1-J)^2}{2\sigma_0}.\]

\begin{lemma}
There exists a unique function $\psi_n$ such that
$$\psi_n\in\{f\in X^{T^*},\ 0\leq f\leq1\}$$ which satisfies
\begin{align*}
\psi_n(s,t,v)&=\chi_n(v\e^{s-t})\,\e^{\frac{\sigma_0}{1-J}(s-t)}+\int_s^t\bigg(\frac{\sigma_0}{1-J}-\frac{\sigma_0}{1-J\mu\psi_n(0,\tau,\cdot)}\bigg)\,\e^{\frac{\sigma_0}{1-J}(s-\tau)}\psi_n(\tau,t,v\e^{s-\tau})\,d\tau\\
&+\int_s^t\frac{\sigma_0}{1-J\mu\psi_n(0,\tau,\cdot)}\,\e^{\frac{\sigma_0}{1-J}(s-\tau)}\big[\psi_n(\tau,t,v\e^{s-\tau}+h)(1-\chi_n(v\e^{s-\tau}))+\psi_n(\tau,t,v_r)\chi_n(v\e^{s-\tau})\big]\,d\tau.
\end{align*}
\end{lemma}


\begin{proof}
Let $T$ be an element of $(0,T^*).$
We use the Banach fixed point theorem for the mapping
\begin{align*}
\Gamma f(s,t,v)=&\,\chi_n(v\e^{s-t})\,\e^{\frac{\sigma_0}{1-J}(s-t)}+\int_s^t\bigg(\frac{\sigma_0}{1-J}-\frac{\sigma_0}{1-J\mu f(0,\tau,\cdot)}\bigg)\,\e^{\frac{\sigma_0}{1-J}(s-\tau)}f(\tau,t,v\e^{s-\tau})\,d\tau\\
&+\int_s^t\frac{\sigma_0}{1-J\mu f(0,\tau,\cdot)}\,\e^{\frac{\sigma_0}{1-J}(s-\tau)}\big[f(\tau,t,v\e^{s-\tau}+h)(1-\chi_n(v\e^{s-\tau}))+f(\tau,t,v_r)\chi_n(v\e^{s-\tau})\big]\,d\tau
\end{align*}
on the invariant complete metric space $\{f\in X^T,\ 0\leq f\leq1\}.$
This mapping is a contraction 
since
\begin{align*}
\|\Gamma f_1-\Gamma f_2\|_\infty&\leq\bigg[\frac{\sigma_0}{1-J}\|f_1-f_2\|_\infty+2\bigg\|\frac{\sigma_0}{1-J\mu f_1}-\frac{\sigma_0}{1-J\mu f_2}\bigg\|_\infty
+\bigg\|\frac{\sigma_0}{1-J\mu f_2}\bigg\|_\infty\|f_1-f_2\|_\infty\bigg]T\\
&\leq\frac{2\,\sigma_0}{(1-J)^2}\,T\,\|f_1-f_2\|_\infty=\frac{T}{T^*}\,\|f_1-f_2\|_\infty.
\end{align*}
\end{proof}

In a second step we define for any $f_0$ belonging to $ C([0,1])$ the function $f$ element of $ X^{T^*}$ as the unique solution to the linear equation
\[\partial_sf(s,t,v)=v\partial_vf(s,t,v)+\frac{\sigma_0}{1-J\mu\psi_n(0,s,\cdot)}\big[f(s,t,v)-f(s,t,v+h)(1-\chi_n(v))-f(s,t,v_r)\chi_n(v)\big]\]
with the terminal condition $$f(t,t,v)=f_0(v).$$
This definition is made more precise in the following lemma.

\begin{lemma}\label{lm:N^n_{s,t}}
For all $f_0$ belonging to $C([0,1])$ there exists a unique $f$ element of $X^{T^*}$ which verifies
\begin{align*}
f(s,t,v)&=f_0(v\e^{s-t})\,\e^{\frac{\sigma_0}{1-J}(s-t)}+\int_s^t\bigg(\frac{\sigma_0}{1-J}-\frac{\sigma_0}{1-J\mu\psi_n(0,\tau,\cdot)}\bigg)\,\e^{\frac{\sigma_0}{1-J}(s-\tau)}f(\tau,t,v\e^{s-\tau})\,d\tau\\
&+\int_s^t\frac{\sigma_0}{1-J\mu\psi_n(0,\tau,\cdot)}\,\e^{\frac{\sigma_0}{1-J}(s-\tau)}\big[f(\tau,t,v\e^{s-\tau}+h)(1-\chi_n(v\e^{s-\tau}))+f(\tau,t,v_r)\chi_n(v\e^{s-\tau})\big]\,d\tau.
\end{align*}
Additionally if $f_0$ is nonnegative then $f$ is too.
\end{lemma}

This allows to define a positive semigroup $(N^n_{s,t})_{0\leq s\leq t<T^*}$ on $C([0,1])$ by $$N^n_{s,t}f_0(v):=f(s,t,v).$$
The semigroup property means that
\[N^n_{t,t}f=f\qquad \text{and}\qquad \forall\tau\in[s,t],\ N^n_{s,t}f=N^n_{s,\tau}(N^n_{\tau,t}f).\]
It is a consequence of the uniqueness of the solution for the integral equation above at time $t$ fixed.
For any $\mathbf t> t>0$ fixed, both $(s,v)\mapsto N^n_{s,\mathbf t}f(v)$ and $(s,v)\mapsto N^n_{s,t}(N^n_{t,\mathbf t}f)(v)$ are solution on $[0,t]\times[0,1]$ with $f_0=N^n_{t,\mathbf t}f$ and the uniqueness property thus yields the equality of these two functions.
Moreover we easily check that $$N^n_{s,t}\1=\1$$ and, together with the positivity property, this ensures the contraction property
\[\|N^n_{s,t}f\|_\infty\leq\|f\|_\infty.\]
A fundamental remark here is that the uniqueness in Lemma~\ref{lm:N^n_{s,t}} ensures that
\[\psi_n(s,t,v)=N^n_{s,t}\chi_n(v).\]
For all $t$ taken in $[0,T^*)$ we define
\[\sigma_n(t):=\frac{\sigma_0}{1-J\mu\psi_n(0,t,\cdot)}=\frac{\sigma_0}{1-J\mu N^n_{0,t}\chi_n}\]
and we denote by $\A^n_t$ the operator defined on $C^1([0,1])$ by
\[\A^n_tf(v):=-vf'(v)+\sigma_n(t)\,\B^nf(v),\]
where $\B^n$ is the regularized jump operator defined on $C([0,1])$ by
\[\B^n f(v):=f(v+h)(1-\chi_n(v))+f(v_r)\chi_n(v)-f(v).\]
The operator $\A_t^n$ is the infinitesimal generator of the semigroup $N^n_{s,t}$ in the sense of the following lemma, where we have set
\[T^{**}:=\frac{(1-J)^2(1-2h)}{3\sigma_0}.\]

\begin{lemma}\label{lm:A^n_t}
If $f$ is an element of $ C^1([0,1]),$ then the function $$(s,t,v)\mapsto N^n_{s,t}f(v)$$ is continuously differentiable on the set $\{0\leq s\leq t\leq T^{**},\ 0\leq v\leq1\}$ and we have
\[\forall\, 0\leq s\leq t\leq T^{**},\qquad\partial_sN^n_{s,t}f=-\A^n_sN^n_{s,t}f,\quad\text{and}\quad\partial_tN^n_{s,t}f=N^n_{s,t}\A^n_tf.\]
\end{lemma}

\begin{proof}[Proof of Lemmas~\ref{lm:N^n_{s,t}} and~\ref{lm:A^n_t}.]
First let $f_0$ be an element of $ C([0,1]),$ $0<T<T^*,$ and define on $X^T$ the mapping
\begin{align*}
\Gamma f(s,t,v)=\ &f_0(v\e^{s-t})\,\e^{\frac{\sigma_0}{1-J}(s-t)}+\int_s^t\bigg(\frac{\sigma_0}{1-J}-\frac{\sigma_0}{1-J\mu\psi_n(0,\tau,\cdot)}\bigg)\,\e^{\frac{\sigma_0}{1-J}(s-\tau)}f(\tau,t,v\e^{s-\tau})\,d\tau\\
&+\int_s^t\frac{\sigma_0}{1-J\mu\psi_n(0,\tau,\cdot)}\,\e^{\frac{\sigma_0}{1-J}(s-\tau)}\big[f(\tau,t,v\e^{s-\tau}+h)(1-\chi_n(v\e^{s-\tau}))+f(\tau,t,v_r)\chi_n(v\e^{s-\tau})\big]\,d\tau.
\end{align*}
For any $f_1,f_2$ belonging to $ X^T$ we have
\[\|\Gamma f_1-\Gamma f_2\|_\infty\leq\min\bigg\{\frac{2\sigma_0}{1-J},\frac{\sigma_0}{(1-J)^2}\bigg\}\|f_1-f_2\|_\infty T\]
and this ensures that $\Gamma$ is a contraction on $X^T$ endowed with the supremum norm.
We deduce the existence and uniqueness of a fixed point for $\Gamma$ from the Banach fixed point theorem.

If $f_0\geq0$ the positive cone of $X^T$ is invariant under $\Gamma$ so the fixed point belongs to this cone. 

\medskip

Now assume that $T<T^{**}$ and $f_0$ belongs to $ C^1([0,1]).$
In this case we can apply the Banach fixed point theorem in the space
$\{f\in X^T:\, \partial_vf,\partial_tf\in X^T\}$ with the norm $$\|f\|_{C^1}:=\|f\|_\infty+\|\partial_vf\|_\infty+\|\partial_tf\|_\infty.$$
Indeed, computing
\begin{align*}
\partial_t\Gamma f(s,t,v)&=\,\A^n_tf_0(v\e^{s-t})\,\e^{\frac{\sigma_0}{1-J}(s-t)}+\int_s^t\bigg(\frac{\sigma_0}{1-J}-\frac{\sigma_0}{1-J\mu\psi_n(0,\tau,\cdot)}\bigg)\,\e^{\frac{\sigma_0}{1-J}(s-\tau)}\partial_tf(\tau,t,v\e^{s-\tau})\,d\tau\\
&+\int_s^t\frac{\sigma_0}{1-J\mu\psi_n(0,\tau,\cdot)}\,\e^{\frac{\sigma_0}{1-J}(s-\tau)}\big[\partial_tf(\tau,t,v\e^{s-\tau}+h)(1-\chi_n(v\e^{s-\tau}))+\partial_tf(\tau,t,v_r)\chi_n(v\e^{s-\tau})\big]\,d\tau
\end{align*}
\begin{align*}
\partial_v\Gamma f(s,t,v)=\ &\,\e^{s-t}f_0'(v\e^{s-t})\,\e^{\frac{\sigma_0}{1-J}(s-t)}+\int_s^t\bigg(\frac{\sigma_0}{1-J}-\frac{\sigma_0}{1-J\mu\psi_n(0,\tau,\cdot)}\bigg)\,\e^{\frac{\sigma_0}{1-J}(s-\tau)}\e^{s-\tau}\partial_vf(\tau,t,v\e^{s-\tau})\,d\tau\\
&+\int_s^t\frac{\sigma_0}{1-J\mu\psi_n(0,\tau,\cdot)}\,\e^{\frac{\sigma_0}{1-J}(s-\tau)}\e^{s-\tau}\partial_vf(\tau,t,v\e^{s-\tau}+h)(1-\chi_n(v\e^{s-\tau}))\\
&+\int_s^t\frac{\sigma_0}{1-J\mu\psi_n(0,\tau,\cdot)}\,\e^{\frac{\sigma_0}{1-J}(s-\tau)}\e^{s-\tau}\frac{n}{h}\1_{[1-h-\frac hn,1-h]}(v\e^{s-\tau})\big[f(\tau,t,v_r)-f(\tau,t,v\e^{s-\tau}+h)\big]\,d\tau
\end{align*}
we get for $f_1,f_2$ taken from $\{f\in X^T:\ \partial_vf,\partial_tf\in X^T\}$
\[\|\partial_t\Gamma f_1-\partial_t\Gamma f_2\|_\infty\leq\min\bigg\{\frac{2\sigma_0}{1-J},\frac{\sigma_0}{(1-J)^2}\bigg\}\|\partial_t(f_1-f_2)\|_\infty T\]
\[\|\partial_v\Gamma f_1-\partial_v\Gamma f_2\|_\infty\leq\bigg[\min\bigg\{\frac{2\sigma_0}{1-J},\frac{\sigma_0}{(1-J)^2}\bigg\}\|\partial_v(f_1-f_2)\|_\infty+\frac{2\sigma_0}{1-J}\frac{1}{1-2h}\|f_1-f_2\|_\infty\bigg]T\]
and finally
\[\|\Gamma f_1-\Gamma f_2\|_{C^1}\leq\min\bigg\{\frac{4\sigma_0}{(1-J)(1-2h)},\frac{3\sigma_0}{(1-J)^2(1-2h)}\bigg\}T\|f_1-f_2\|_{C^1}.\]
We deduce that the unique fixed point of $\Gamma$ satisfies $$\partial_vf,\partial _tf\in X^T.$$
We can also compute
\begin{align*}
\partial_s\Gamma f(s,t,v)
=\ &v\partial_v\Gamma f(s,t,v)+\frac{\sigma_0}{1-J}\big[\Gamma f(s,t,v)-f(s,t,v)\big]\\
&\hspace{14mm}+\frac{\sigma_0}{1-J\mu\psi_n(0,s,\cdot)}\big[f(s,t,v)-f(s,t,v+h)(1-\chi_n(v))-f(s,t,v_r)\chi_n(v)\big]
\end{align*}
and this ensures that the fixed point also satisfies $$\partial_sf\in X^T\qquad\text{and}\qquad \partial_sf=-\A^n_sf.$$

From the computation of $\partial_t\Gamma f$ we see that if $N^n_{s,t}f_0$ is the fixed point of $\Gamma$ with terminal condition $f_0$ then $\partial_tN^n_{s,t}f_0$ is the fixed of $\Gamma$ with terminal condition $\A^n_tf_0.$
By uniqueness we deduce that $$\partial_tN^n_{s,t}f_0=N^n_{s,t}\A^n_tf_0.$$

\end{proof}

The third step consists in defining the measure $\mu N^n_{s,t}$ by duality
\[\forall f\in C([0,1]),\qquad(\mu N^n_{s,t})f:=\mu (N^n_{s,t}f).\]
The following lemma ensures that $t\mapsto\mu N^n_{0,t}$ is a solution to a regularized version of Equation~\eqref{eq:LIF} on the interval $[0,T^{**}).$

\begin{lemma}\label{lm:reg}
The mapping $$t\mapsto\mu N^n_{0,t},$$ which is defined on $[0,T^*),$ takes its values in $\P([0,1])$ and is weak*-continuous.
Additionally for all $t$ in $[0,T^{**})$ and all $f$ in $C^1([0,1])$ we have
\begin{equation}\label{eq:nonlin_reg}
\mu N^n_{0,t}f=\mu f+\int_0^t\mu N^n_{0,s}\A^n_sf\,ds.
\end{equation}
\end{lemma}

\begin{proof}
The positivity property in Lemma~\ref{lm:N^n_{s,t}} ensures that $\mu N^n_{0,t}$ belongs to $\M_+([0,1]).$
Additionally we easily check that $$N^n_{0,t}\1=\1.$$
Together with the positivity this implies that $\mu N^n_{0,t}$ is an element of $\P([0,1]),$ and also that $$\|N^n_{0,t}f\|_\infty\leq\|f\|_\infty$$ for all $f$ belonging to $C([0,1]).$
The weak*-continuity of the mapping $$t\mapsto\mu N^n_{0,t}$$ follows from the continuity of $$t\mapsto N^n_{0,t}f(v)$$ for all $f$ belonging to $C([0,1]),\,v\in[0,1],$ and from the dominated convergence theorem.

\medskip

For~\eqref{eq:nonlin_reg} we prove a little bit more, namely that for all $f$ belonging to $ C^1([0,1])$ the mapping $$t\mapsto\mu N^n_{0,t}f$$ is continuously differentiable and that
\[\frac{d}{dt}(\mu N^n_{0,t}f)=\mu N^n_{0,t}\A^n_tf.\]
Indeed, from Lemma~\ref{lm:A^n_t} and by dominated convergence we have
\[\frac1h\big(\mu N^n_{0,t+h}f-\mu N^n_{0,t}f\big)=\mu\Big[\frac1h\big(N^n_{0,t+h}f-N^n_{0,t}f\big)\Big]\xrightarrow[h\to0]{}\mu(\partial_t N^n_{0,t}f)=\mu N^n_{0,t}\A^n_tf.\]


\end{proof}

In the fourth step we pass to the limit $n$ goes to infinity. 

\begin{lemma}\label{lm:local}
For all $t\in[0,T^*),$ the sequence $(\mu N^n_{0,t})_{n\in\N^*}$ is convergent for the total variation norm.
Denoting $(\mu_t)_{0\leq t<T^*}\subset\P([0,1])$ the limit family, we have for all $t\in[0,T^{**})$ and all $f\in C^1([0,1])$
\[\mu_tf=\mu f+\int_0^t\mu_s\A_{\sigma(s)}f\,ds.\]
\end{lemma}

\begin{proof}
We check that $(\mu N^n_{0,t})_{n\in\N^*}$ is a Cauchy sequence.
Let $n,p$ two elements of $\N^*,$ $0\leq s\leq t<T^*,$ and $f$ taken from $ C([0,1])$ such that $\|f\|_\infty\leq1.$
We have, using that $\mu$ belongs to $\P([0,1])$ and the Fubini's theorem,
\begin{align*}
\|N^n_{s,t}f-N^{n+p}_{s,t}f\|_\infty&\leq2\int_s^t\bigg|\frac{\sigma_0}{1-J\mu N^n_{0,\tau}\chi_n}-\frac{\sigma_0}{1-J\mu N^{n+p}_{0,\tau}\chi_{n+p}}\bigg|\,d\tau
+\frac{2\sigma_0}{1-J}\int_s^t\|N^n_{\tau,t}f-N^{n+p}_{\tau,t}f\|_\infty\,d\tau\\
&\leq\frac{2\sigma_0J}{(1-J)^2}\int_s^t|\mu N^n_{0,\tau}\chi_n-\mu N^{n+p}_{0,\tau}\chi_{n+p}|\,d\tau
+\frac{2\sigma_0}{1-J}\int_s^t\|N^n_{\tau,t}f-N^{n+p}_{\tau,t}f\|_\infty\,d\tau\\
&\leq\frac{2\sigma_0J}{(1-J)^2}\int_s^t|\mu N^n_{0,\tau}\chi_n-\mu N^n_{0,\tau}\chi_{n+p}|\,d\tau\\
&\qquad+\frac{2\sigma_0J}{(1-J)^2}\int_s^t|\mu N^n_{0,\tau}\chi_{n+p}-\mu N^{n+p}_{0,\tau}\chi_{n+p}|\,d\tau
+\frac{2\sigma_0}{1-J}\int_s^t\|N^n_{\tau,t}f-N^{n+p}_{\tau,t}f\|_\infty\,d\tau\\
&\leq\frac{2\sigma_0J}{(1-J)^2}\sup_{v\in[0,1]}\int_0^t |N^n_{0,\tau}(\chi_n-\chi_{n+p})(v)|\,d\tau\\
&\qquad+\frac{2\sigma_0}{(1-J)^2}\int_s^t\left\|\mu N^n_{\tau,t}-\mu N^{n+p}_{\tau,t}\right\|_{\mathrm{TV}}\,d\tau
+\frac{2\sigma_0}{1-J}\int_s^t\|N^n_{\tau,t}f-N^{n+p}_{\tau,t}f\|_\infty\,d\tau.
\end{align*}
We give an estimate on the quantity
\[\Omega_n(s,t):=\sup_{p\in\N}\sup_{v\in[0,1]}\int_s^t|N^n_{s,\tau}(\chi_n-\chi_{n+p})(v)|\,d\tau, \quad n\in\N^*, \quad 0\leq s\leq t<T^*.\] 
From the definition of the semigroup $(N^n_{s,t})$ we have for all $v$ taken from $[0,1],$ $n$ from $\N^*,$ and $0\leq s\leq t<T^*$
\begin{align*}
\int_s^t |N^n_{s,\tau}(\chi_n-\chi_{n+p})(v)|\,d\tau&\leq\int_s^t|(\chi_n-\chi_{n+p})(v\e^{s-\tau})|\,d\tau
+\frac{\sigma_0}{1-J}\int_s^t\int_s^\tau|N^n_{\tau',\tau}(\chi_n-\chi_{n+p})(v\e^{s-\tau'})|\,d\tau' d\tau\\
&\hspace{-34mm}+\frac{\sigma_0}{1-J}\int_s^t\int_s^\tau|N^n_{\tau',\tau}(\chi_n-\chi_{n+p})(v\e^{s-\tau'}+h)(1-\chi_n(v\e^{s-\tau'})+N^n_{\tau',\tau}(\chi_n-\chi_{n+p})(v_r)\chi_n(v\e^{s-\tau'})|\,d\tau' d\tau\\
&\leq\int_s^t\1_{[1-h-\frac hn,1-h]}(v\e^{s-\tau})\,d\tau
+\frac{\sigma_0}{1-J}\int_s^t\int_{\tau'}^t|N^n_{\tau',\tau}(\chi_n-\chi_{n+p})(v\e^{s-\tau'})|\,d\tau d\tau'\\
&\hspace{-34mm}+\frac{\sigma_0}{1-J}\int_s^t\int_{\tau'}^t|N^n_{\tau',\tau}(\chi_n-\chi_{n+p})(v\e^{s-\tau'}+h)(1-\chi_n(v\e^{s-\tau'})+N^n_{\tau',\tau}(\chi_n-\chi_{n+p})(v_r)\chi_n(v\e^{s-\tau'})|\,d\tau d\tau'\\
&\leq\log\bigg(1+\frac{h}{n(1-2h)}\bigg)+\frac{3\sigma_0}{1-J}\int_s^t\Omega_n(\tau',t)\,d\tau'.
\end{align*}
Taking the supremum in the left hand side we get the inequality
\[\Omega_n(s,t)\leq\log\bigg(1+\frac{h}{n(1-2h)}\bigg)+\frac{3\sigma_0}{1-J}\int_s^t\Omega_n(\tau,t)\,d\tau\]
which gives by Gr\"onwall's lemma
\[\Omega_n(s,t)\leq\log\bigg(1+\frac{h}{n(1-2h)}\bigg)\e^{\frac{3\sigma_0}{1-J}(t-s)}.\]
Coming back to the first computations of the proof we get
\begin{align*}
\|N^n_{s,t}f-N^{n+p}_{s,t}f\|_\infty&\leq\frac{2\sigma_0J}{(1-J)^2}\log\bigg(1+\frac{h}{n(1-2h)}\bigg)\e^{\frac{3\sigma_0}{1-J}t}
+\frac{2\sigma_0}{(1-J)^2}\int_s^t\left\|\mu N^n_{\tau,t}-\mu N^{n+p}_{\tau,t}\right\|_{\mathrm{TV}}\,d\tau\\
&\qquad+\frac{2\sigma_0}{1-J}\int_s^t\|N^n_{\tau,t}f-N^{n+p}_{\tau,t}f\|_\infty\,d\tau
\end{align*}
which gives by Gr\"onwall's lemma
\[\|N^n_{s,t}f-N^{n+p}_{s,t}f\|_\infty\leq\frac{2\sigma_0}{(1-J)^2}\bigg[J\log\bigg(1+\frac{h}{n(1-2h)}\bigg)\e^{\frac{3\sigma_0}{1-J}t}
+\int_s^t\left\|\mu N^n_{\tau,t}-\mu N^{n+p}_{\tau,t}\right\|_{\mathrm{TV}}\,d\tau\bigg]\e^{\frac{2\sigma_0}{1-J}(t-s)}.\]
Finally
\begin{align*}
\left\|\mu N^n_{s,t}-\mu N^{n+p}_{s,t}\right\|_{\mathrm{TV}}&\leq\sup_{\|f\|_\infty\leq1}\|N^n_{s,t}f-N^{n+p}_{s,t}f\|_\infty\\
&\leq\frac{2\sigma_0\e^{\frac{2\sigma_0}{1-J}t}}{(1-J)^2}\bigg[J\log\bigg(1+\frac{h}{n(1-2h)}\bigg)\e^{\frac{3\sigma_0}{1-J}t}
+\int_s^t\left\|\mu N^n_{\tau,t}-\mu N^{n+p}_{\tau,t}\right\|_{\mathrm{TV}}\,d\tau\bigg]
\end{align*}
and by Gr\"onwall's lemma
\[\left\|\mu N^n_{s,t}-\mu N^{n+p}_{s,t}\right\|_{\mathrm{TV}}\leq\frac{2\sigma_0J}{(1-J)^2}\log\bigg(1+\frac{h}{n(1-2h)}\bigg)\exp\bigg(\frac{5\sigma_0}{1-J}t+\frac{2\sigma_0\e^{\frac{2\sigma_0}{1-J}t}}{(1-J)^2}(t-s)\bigg).\]
We deduce that for all $t\in[0,T^*)$ the sequence $(\mu N^n_{0,t})_{n\in\N^*}\subset\P([0,1])$ is a Cauchy sequence, hence convergent to a limit $\mu_t\in\P([0,1]),$ and additionally
\[\sup_{0\leq t<T^*}\left\|\mu N^n_{0,t}-\mu_t\right\|_{\mathrm{TV}}\leq\frac{2\sigma_0J}{(1-J)^2}\log\bigg(1+\frac{h}{n(1-2h)}\bigg)\exp\bigg(\frac{5\sigma_0}{1-J}T^*+\frac{2\sigma_0\e^{\frac{2\sigma_0}{1-J}T^*}}{(1-J)^2}T^*\bigg).\]
This allows us to pass to the limit in Lemma~\ref{lm:reg}, and the proof is complete.
\end{proof}

The last step consists in proving that any solution to Equation~\eqref{eq:LIF} satisfies a Duhamel formula.

\begin{lemma}
Let $(\mu_t)_{0\leq t<T}$ be a solution to Equation~\eqref{eq:LIF}.
Then for any $\sigma>0$ the following Duhamel formula is verified
\begin{equation}\label{eq:Duhamel}
\forall f\in C([0,1]),\,\forall t\geq0,\qquad\mu_tf=\mu_0 M_tf+\int_0^t(\sigma(s)-\sigma)\mu_s\B M_{t-s}f\,ds
\end{equation}
where $(M_t)_{t\geq0}$ is the semigroup generated by $\A_\sigma$ (see Section~\ref{sec:linear_case}).
\end{lemma}

\begin{proof}
Using the semigroup $(M^n_t)_{t\geq0}$ defined in Appendix~\ref{app:linear} we have
\begin{align*}
\frac{d}{ds}\bigg(\int_0^s\mu_\tau M^n_{t-s}f\,d\tau\bigg)&=\mu_sM^n_{t-s}f-\int_0^s\mu_\tau\A^nM^n_{t-s}f\,d\tau\\
&=\mu_0 M^n_{t-s}f+\int_0^s\mu_\tau(\A_{\sigma(\tau)}-\A^n)M^n_{t-s}f\,d\tau\\
&=\mu_0 M^n_{t-s}f+\int_0^s\mu_\tau(\sigma(\tau)\B-\sigma\B^n)M^n_{t-s}f\,d\tau.
\end{align*}
Integrating between $0$ and $t$ we get
\[\int_0^t\mu_\tau f\,d\tau=\int_0^t\mu_0 M^n_{s}f\,ds+\int_0^t\int_0^s\mu_\tau(\sigma(\tau)\B-\sigma\B^n)M^n_{t-s}f\,d\tau\,ds.\]
Differentiating with respect to $t$ we obtain (by using dominated convergence, Fubini's theorem, and a change of variable)
\begin{align*}
\mu_tf&=\mu_0 M^n_{t}f+\int_0^t\mu_r(\sigma(r)\B-\sigma\B^n)f\,dr+\int_0^t\int_0^s\mu_r(\sigma(r)\B-\sigma\B^n)\A^nM^n_{t-s}f\,dr\,ds\\
&=\mu_0 M^n_{t}f+\int_0^t\mu_r(\sigma(r)\B-\sigma\B^n)f\,dr+\int_0^t\mu_r(\sigma(r)\B-\sigma\B^n)\bigg(\int_r^t\A^nM^n_{t-s}f\,ds\bigg)dr\\
&=\mu_0 M^n_{t}f+\int_0^t\mu_r(\sigma(r)\B-\sigma\B^n)\bigg(f+\int_0^{t-r}\A^nM^n_{s}f\,ds\bigg)dr\\
&=\mu_0 M^n_{t}f+\int_0^t\mu_r(\sigma(r)\B-\sigma\B^n)M^n_{t-r}f\,dr
\end{align*}
 and then passing to the limit $n\to\infty$
\[\mu_tf=\mu_0 M_{t}f+\int_0^t(\sigma(r)-\sigma)\mu_r\B M_{t-r}f\,dr.\]
\end{proof}

We are now ready to prove Theorem~\ref{th:globalwp}.

\begin{proof}[Proof of Theorem~\ref{th:globalwp}]
With Lemma~\ref{lm:local} we have proved for any $\mu_0$ element of $\P([0,1])$ the existence of a local solution, on $[0,T^{**}).$
But since $T^{**}$ is independent of $\mu_0,$ we can iterate the procedure to get a global solution.
The uniqueness is a consequence of the Duhamel formula~\eqref{eq:Duhamel}.
Let $(\mu_t)_{0\leq t<T}$ and $(\widetilde\mu_t)_{0\leq t<\widetilde T}$ be two solutions.
Then for any $t$ in $[0,\min\{T,\widetilde T\})$ we get from the Duhamel formula with $\sigma$ being taken to be equal to $\frac{\sigma_0}{1-J}$
\begin{align*}
\left\|\mu_t-\widetilde\mu_t\right\|_{\mathrm{TV}}&
\leq\left\|\mu_0-\widetilde\mu_0\right\|_{\mathrm{TV}}+2\int_0^t|\sigma(s)-\widetilde\sigma(s)|\,ds+\frac{2\sigma_0}{1-J}\int_0^t\left\|\mu_s-\widetilde\mu_s\right\|_{\mathrm{TV}}\,ds\\
&\leq\left\|\mu_0-\widetilde\mu_0\right\|_{\mathrm{TV}}+\frac{2\sigma_0}{(1-J)^2}\int_0^t\left\|\mu_s-\widetilde\mu_s\right\|_{\mathrm{TV}}\,ds
\end{align*}
and by Gr\"onwall's lemma
\[\left\|\mu_t-\widetilde\mu_t\right\|_{\mathrm{TV}}\leq\left\|\mu_0-\widetilde\mu_0\right\|_{\mathrm{TV}}\,\e^{\frac{2\sigma_0}{(1-J)^2}t}.\]
\end{proof}

\section{Steady state analysis}\label{sec:steadystates}

This section is devoted to the proof of Theorem~\ref{th:steadystate}.
The result in Theorem~\ref{th:main_lin} ensures that for all $\sigma>0$ there exists a unique $\mu^\sigma$ belonging to $\P([0,1])$ such that $$\mu^\sigma\A_\sigma=0.$$
The question we address here is the existence of $\bar\sigma>0$ such that
\begin{equation}\label{eq:nonlin_steadystate}
\bar\sigma=\frac{\sigma_0}{1-J\mu^{\bar\sigma}([1-h,1])}.
\end{equation}
In this case the measure $\bar\mu:=\mu^{\bar\sigma}$ is a steady state for the nonlinear equation.
We will prove the existence of such a steady state when
$$J<1+\big\lfloor\frac{1-v_r}{h}\big\rfloor$$
and the existence of at least two steady states when
\[J>1+\left\lfloor\frac{1-v_r}{h}\right\rfloor\qquad\text{and}\qquad\sigma_0<\frac{1-h}{4J}.\]
These conditions have to be compared to the condition
\[J\geq1+\frac{1-v_r}{h}\qquad\text{and}\qquad h\sigma_0>1\]
which ensures that all the solutions to Equation~\eqref{eq:LIF} blow up in finite time whatever their initial distribution~\cite{DH02}.

\medskip

For finding $\bar\sigma$ which satisfies~\eqref{eq:nonlin_steadystate}, we define the two functions
\[F\,:\left\{\begin{array}{ccl}
(0,+\infty) & \to & [0,+\infty)
\vspace{1mm}\\
\sigma & \mapsto &\displaystyle \mu^\sigma([1-h,1])
\end{array}\right.
\qquad\text{and}\qquad
G\,:
\left\{\begin{array}{ccl}
(0,+\infty) & \to & \R \\
\sigma & \mapsto & \dfrac{1}{J}\left(1-\dfrac{\sigma_0}{\sigma} \right)
\end{array}\right..\]
and we prove the existence of $\bar\sigma$ such that $$F(\bar\sigma)=G(\bar\sigma).$$
To do so we need informations on the function $F$ and it requires some regularity results on the invariant measure $\mu^\sigma.$

\begin{lemma}\label{lm:p_sigma}
For any $\sigma>0,$ the invariant measure $\mu^\sigma$ is absolutely continuous with respect to the Lebesgue measure.
Its density $p_\sigma$ satisfies 
\[vp_\sigma(v)\in W^{1,\infty}([0,v_r)\cup(v_r,1]), \quad \forall v\in(0,1),\qquad 0<p_\sigma(v)\leq\min\Big\{\frac{\sigma v^{\sigma-1}}{h^\sigma},\frac\sigma v\Big\}.\]

\end{lemma}

\begin{proof}
Recall that $\mu^\sigma$ satisfies
\begin{equation}\label{eq:weak_steady}
\int_{[0,1]}\Big[\frac v\sigma f'(v)+f(v)-f(v+h)\1_{[0, 1-h)}(v)\Big]\,d\mu^\sigma(v)=f(v_r)\mu^\sigma([1-h,1]),\qquad\forall f\in C^1([0,1]).
\end{equation}
Thus the derivative of $v\mu^\sigma(dv)$ in the distributional sense is a finite measure and then $v\mu^\sigma(dv)$ is a function with bounded variation.
We deduce that there exist $\alpha\geq0$ and $p_\sigma\in L^1_+(0,1)$ such that 
$$\mu^\sigma=\alpha\,\delta_0+p_\sigma(v)\,dv.$$
More precisely $vp_\sigma(v)$ is a $W^{1,1}$ function on the intervals $(0,h),$ $(h,v_r)$ and $(v_r,1)$ (so that it has a left and a right trace at $v=h$ and $v=v_r$) with a jump at $v=v_r$ given by
\[v_rp_\sigma(v_r^-)-v_rp_\sigma(v_r^+)=\sigma\int_{1-h}^1 p_\sigma(w)\,dw\]
and also {\it a priori} at $v=h$ given by
\[hp_\sigma(h^-)-hp_\sigma(h^+)=\mu^\sigma(\{0\})=\alpha.\]
We will actually prove that $\mu^\sigma$ does not charge $0,$ {\it i.e.} $\alpha=0,$ so that there is no jump at $v=h.$
Consider $f\in C^1_c([0,1))$ which satisfies $f(0)=1,$
and define $f_n(v):=f(nv).$
For all $n>\lfloor1/h\rfloor$ Equation~\eqref{eq:weak_steady} written with $f_n$ gives
\[\int_{[0,1]}\Big[\frac v\sigma f_n'(v)+f_n(v)\Big]\,d\mu^\sigma(v)=0.\]
By dominated convergence we have
\[\int_{[0,1]}f_n(v)\,d\mu^\sigma(v)\xrightarrow[n\to\infty]{}\mu^\sigma(\{0\})=\alpha\]
and
\[\int_{[0,1]}vf_n'(v)\,d\mu^\sigma(v)=\int_{[0,1]}nvf'(nv)\,d\mu^\sigma(v) 
\xrightarrow[n\to\infty]{}0.\]
We conclude that $\alpha=0,$ so that $\mu^\sigma=p_\sigma(v)\,dv.$
The function $p_\sigma$ has no jump at $v=h$ and it satisfies the following equation on $(0,v_r)\cup(v_r,1)$
\[\frac{-1}{\sigma}(vp_\sigma(v))'+p_\sigma(v)=p_\sigma(v-h)\1_{[0,1-h)}(v).\]

\

For the bound on $p_\sigma$ we start by studying $p_\sigma$ on the interval $(0,h).$
On this interval the equation satisfied by $p_\sigma$ is
\[(vp_\sigma(v))'=p_\sigma(v)\]
which gives after integration
\begin{equation}\label{eq:explicit[0,h]}
p_\sigma(v)=\Big(\frac vh\Big)^{\sigma-1}p_\sigma(h^-).
\end{equation}

\

By positivity of $p_\sigma$ we have on both intervals $(0,v_r)$ and $(v_r,1)$ the differential inequality
\[(vp_\sigma(v))'\leq\sigma p_\sigma(v)\]
from which we get
\begin{equation}\label{eq:inequality-p_sigma}
\forall\ 0<w<v\leq1,\qquad p_\sigma(v)\leq\Big(\frac{v}{w}\Big)^{\sigma-1}p_\sigma(w).
\end{equation}
Integrating from $w=0$ to $w=v$ we deduce
\[1\geq\int_0^{v}p_\sigma(w)\,dw\geq\frac{p_\sigma(v)}{v^{\sigma-1}}\int_0^{v}w^{\sigma-1}dw=\frac{v p_\sigma(v)}{\sigma}\]
and
\[p_\sigma(v)\leq\frac\sigma v\qquad\forall\ 0<v\leq1.\]
Combining with~\eqref{eq:inequality-p_sigma} we get that
\[\forall\ 0<w<v\leq1,\qquad p_\sigma(v)\leq\frac{\sigma}{w^\sigma} v^{\sigma-1}\]
and with~\eqref{eq:explicit[0,h]} we obtain
\[p_\sigma(v)\leq\frac{\sigma v^{\sigma-1}}{h^\sigma}\qquad\forall\ 0<v\leq1.\]

\

From~\eqref{eq:inequality-p_sigma} we deduce that the support of $p_\sigma$ is necessarily of the form $[0,v_m]$ with $0<v_m\leq1$
and $p_\sigma$ is strictly positive in the interior of its support.
Assume that $v_m<1.$
Then integrating the equation satisfied by $p_\sigma$ between $v_m$ and $\tilde v:=\min\{v_m+h,1\}>v_m$ we get the contradiction
\[0=\int_{v_m-h}^{\tilde v-h}p_\sigma(v)\,dv>0.\]
So $v_m=1$ and $p_\sigma(v)>0$ for all $v$ in $(0,1).$

\end{proof}

Lemma~\ref{lm:p_sigma} allows to get informations about the asymptotic behaviour of the function $F.$

\begin{lemma}\label{lm:Fproperties}
The function $F$ is continuous and satisfies
\[\lim_{\sigma\to0}F(\sigma)=0\qquad\text{and}\qquad
\lim_{\sigma\to+\infty}F(\sigma)= \left(1+ \left\lfloor \dfrac{1-v_r}{h}\right\rfloor \right) ^{-1}.\]
\end{lemma}

\begin{proof}
Fix $\sigma>0$ and consider a sequence $(\sigma_n)_{n\in\N} $ belonging to $[\sigma/2,2\sigma]$ which converges to $\sigma.$
Since the associated sequence $(p_{\sigma_n})_{n\in\N}$ satisfies
\[\forall n\in\N,\,\forall v\in[0,1],\qquad 0\leq p_{\sigma_n}(v)\leq\frac{2\sigma v^{\sigma/2-1}}{h^{2\sigma}}\]
we deduce from the Dunford-Pettis theorem that there exists a subsequence, still denoted by $(p_{\sigma_n}),$ which converges $L^1$-weak to a limit $q$ element of $\P([0,1])\cap L^1(0,1),$ {\it i.e.}
\[\int_0^1p_{\sigma_n}(v)\varphi(v)\,dv\to\int_0^1 q(v)\varphi(v)\,dv\]
for all $\varphi$ belonging to $L^\infty(0,1).$
Passing to the limit in the weak formulation we get that $q$ is solution to~\eqref{eq:weak_steady}.
By uniqueness we get that $q=p_\sigma$ and the whole sequence converges to $p_\sigma.$
This gives the continuity of the application
\[\sigma\mapsto\int_0^1 p_{\sigma}(v)\varphi(v)\,dv\]
for any $\varphi\in L^\infty(0,1),$ and as a consequence the continuity of $F$ since $\mathbf1_{[1-h,1]}\in L^\infty(0,1).$

\

For the limit at $0$ we readily deduce from Lemma~\ref{lm:p_sigma} that
\[0\leq F(\sigma)\leq\frac{\sigma}{1-h}\xrightarrow[\sigma\to0]{}0.\]
If we want to be more precise we can prove that $p_\sigma$ converges weak* to $\delta_0$ when $\sigma\to0.$
Indeed for any sequence $(\sigma_n)_{n\in\mathds N}$ which tends to $0$
we can extract from $(p_{\sigma_n})_{n\geq0}$ a subsequence, still denoted $(p_{\sigma_n})_{n\geq0},$ which converges weak* to a probability measure $\mu.$
Lemma~\ref{lm:p_sigma} ensures that
\[\forall\, 0<\epsilon\leq v\leq1,\,\forall\sigma>0,\qquad 0\leq p_{\sigma}(v)\leq \frac{\sigma}{\epsilon}\] 
so we deduce that for all $f\in C_c([\epsilon,1])$ we have
\[\mu f=\lim_{n\to\infty}\int_0^1p_{\sigma_n}(v)f(v)\,dv=\lim_{n\to\infty}\int_\epsilon^1p_{\sigma_n}(v)f(v)\,dv=0.\]
As a consequence $\supp\mu={0}$ and since $\mu$ belongs to $\P([0,1])$ we deduce that $\mu=\delta_0,$
and then $p_\sigma\stackrel{*}{\rightharpoonup}\delta_0$ when $\sigma\to0$ since the sequence $(\sigma_n)$ is arbitrary.

\

We finish with the limit at infinity.
Let $(\sigma_n)_{n\in\mathds N}$ a sequence which tends to $+\infty.$
We can extract from $(p_{\sigma_n})_{n\geq0}$ a subsequence, still denoted $(p_{\sigma_n})_{n\geq0},$ which weakly converges to a probability measure $\mu.$
We want to identify the limit $\mu.$
We have that
\[\int_0^1p_{\sigma_n}(v)f(v)\,dv\to\mu f\qquad(n\to\infty)\]
for all $f$ element of $C([0,1]).$
We define
\[D:=\{f\in C^1([0,1]),f(v_r)=f(1)\}\]
which satisfies the property that
\[\forall f\in D,\quad \A f\in C([0,1])\]
and as a consequence
\[\forall f\in D,\qquad 0=\frac{1}{\sigma_n}\mu^{\sigma_n}\A_{\sigma_n}f=\frac{1}{\sigma_n}\mu^{\sigma_n}(vf'(v))+\mu^{\sigma_n}\B f\xrightarrow[n\to\infty]{}\mu\B f.\]
This property that $\mu\B f=0$ for all $f$ taken in $ D,$ {\it i.e.}
\[\int_{[0,1]}\Big[f(v)-f(v+h)\mathbf1_{[0,1-h)}(v)\Big]\,d\mu(v)=f(v_r)\mu([1-h,1]),\qquad\forall f\in D.\]
allows to prove that
\[\mu= \left(1+ \left\lfloor \dfrac{1-v_r}{h}\right\rfloor \right) ^{-1}\sum_{k=0}^{ \left\lfloor \frac{1-v_r}{h}\right\rfloor}\delta_{v_r+kh}.\]
We prove this step by step.
First for $f$ taken from $ C^1_c([0,h))$ we get $$\int_{[0,h)}f(v)d\mu(v)=0$$ and so $$\supp \mu\cap[0,h)=\emptyset.$$
We easily deduce by induction, choosing $f\in C^1_c(kh,(k+1)h),$ that $\supp \mu\cap[0,kh)=0$ for any $k\in\N$ such that $kh\leq v_r$
and then with one more step, with $f\in C^1_c(kh,v_r),$ that $\supp \mu\cap[0,v_r)=\emptyset.$
Keeping going we get by choosing $f $ in $C^1_c(v_r+kh,v_r+(k+1)h)$ that $$\supp \mu\cap(v_r+kh,v_r+(k+1)h)=\emptyset, \quad\forall k\in\N.$$
Finally we have proved that
\[\supp \mu\subset\bigg\{\{v_r+kh\},\ 0\leq k\leq  \left\lfloor \dfrac{1-v_r}{h}\right\rfloor\bigg\}\]
so there exists a finite family of nonnegative real numbers $\alpha_k$ such that
\[\mu=\sum_{k=0}^{ \left\lfloor \frac{1-v_r}{h}\right\rfloor}\alpha_k\,\delta_{v_r+kh}.\]
For $$1\leq k\leq \left\lfloor \dfrac{1-v_r}{h}\right\rfloor,$$ considering $f$ belonging to $ C^1_c(v_r+(k-1/2)h,\min\{v_r+(k+1/2)h,1\})$ such that $f(v_r+kh)=1$ as a test function, we get that $\alpha_k=\alpha_{k-1}.$
Now let $f$ be an element of $ C^1([0,1])$ such that $$f(v_r)=f(1)=1, \quad \supp f\subset[v_r-h/2,v_r+h/2]\cup[v_r+\left\lfloor \frac{1-v_r}{h}\right\rfloor h,1].$$
Using this function as a test function we get that $$\alpha_0=\alpha_{\left\lfloor \frac{1-v_r}{h}\right\rfloor}.$$
Finally all the $\alpha_k$ are equal and
\[\mu= \left(1+ \left\lfloor \dfrac{1-v_r}{h}\right\rfloor \right) ^{-1}\sum_{k=0}^{ \left\lfloor \frac{1-v_r}{h}\right\rfloor}\delta_{v_r+kh}.\]
Since this limit does not depend on the subsequence, we deduce that for any $f$ taken from $ C([0,1])$ we have
\[\lim_{\sigma\to+\infty}\int_0^1f(v)p_\sigma(v)\,dv=\int_0^1f(v)\,d\mu(v)=\left(1+ \left\lfloor \dfrac{1-v_r}{h}\right\rfloor \right) ^{-1}\sum_{k=0}^{ \left\lfloor \frac{1-v_r}{h}\right\rfloor}f(v_r+kh).\]
Since $\mathbf1_{[1-h,1]}$ is not continuous we cannot conclude directly for the limit of $F.$
For all $n\geq1$ we define $$\tilde\chi_n(v):=\chi_n(v-\frac hn).$$
For all $v$ in $[0,1]$ we have $$\tilde\chi_n(v)\leq\1_{[1-h,1]}(v)\leq\chi_n(v)$$ and since $v_r+\left\lfloor \frac{1-v_r}{h}\right\rfloor h$ belongs to $[1-h,1]$
\[\lim_{\sigma\to+\infty}\int_0^1\chi_n(v)p_\sigma(v)\,dv=\left(1+ \left\lfloor \dfrac{1-v_r}{h}\right\rfloor \right) ^{-1}
\bigg[\chi_n\bigg(v_r+\left\lfloor \frac{1-v_r}{h}\right\rfloor h-h\bigg)+\chi_n\bigg(v_r+\left\lfloor \frac{1-v_r}{h}\right\rfloor h\bigg)\bigg],\]
\[\lim_{\sigma\to+\infty}\int_0^1\tilde\chi_n(v)p_\sigma(v)\,dv=\left(1+ \left\lfloor \dfrac{1-v_r}{h}\right\rfloor \right) ^{-1}\tilde\chi_n\bigg(v_r+\left\lfloor \frac{1-v_r}{h}\right\rfloor h\bigg).\]
We deduce that for all $n\geq1$
\[\tilde\chi_n\bigg(v_r+\left\lfloor \frac{1-v_r}{h}\right\rfloor h\bigg)
\leq\left(1+ \left\lfloor \dfrac{1-v_r}{h}\right\rfloor \right)\lim_{\sigma\to+\infty}F(\sigma)
\leq\chi_n\bigg(v_r+\left\lfloor \frac{1-v_r}{h}\right\rfloor h-h\bigg)+\chi_n\bigg(v_r+\left\lfloor \frac{1-v_r}{h}\right\rfloor h\bigg).\]
The fact that $$v_r+\left\lfloor \frac{1-v_r}{h}\right\rfloor h\in(1-h,1),$$ which is guaranteed from the assumptions we made, ensures that
\[\lim_{n\to\infty}\chi_n\bigg(v_r+\left\lfloor \frac{1-v_r}{h}\right\rfloor h\bigg)
=\lim_{n\to\infty}\tilde\chi_n\bigg(v_r+\left\lfloor \frac{1-v_r}{h}\right\rfloor h\bigg)=1\]
and
\[\lim_{n\to\infty}\chi_n\bigg(v_r+\left\lfloor \frac{1-v_r}{h}\right\rfloor h-h\bigg)=0,\]
and the conclusion follows.
\end{proof}

\begin{corollary}
If
\[J<1+\left\lfloor\frac{1-v_r}{h}\right\rfloor,\]
then there exists at least one steady state,
and if
\[J>1+\left\lfloor\frac{1-v_r}{h}\right\rfloor\qquad\text{and}\qquad\sigma_0<\frac{1-h}{4J},\]
then there exist at least two steady states.
\end{corollary}

\begin{proof}
If the condition $$J<1+\big\lfloor\frac{1-v_r}{h}\big\rfloor$$ is satisfied then Lemma~\ref{lm:Fproperties} ensures the existence of $\bar\sigma>0$ such that $$F(\bar\sigma)=G(\bar\sigma).$$
Then $\bar\sigma$ satisfies~\eqref{eq:nonlin_steadystate} and $\mu^{\bar\sigma}$ is a steady state of the nonlinear equation.

When $$J>1+\big\lfloor\frac{1-v_r}{h}\big\rfloor$$ we have
$$\lim_{\sigma\to0}F(\sigma)-G(\sigma)=+\infty\qquad\text{and} \qquad \lim_{\sigma\to+\infty}F(\sigma)-G(\sigma)>0.$$
On the other hand Lemma~\ref{lm:p_sigma} implies that for all $\sigma>0$
\[F(\sigma)-G(\sigma)\leq\frac{\sigma}{1-h}-\frac1J.\]
The minimum of the right hand side function is $\sqrt{\frac{4\sigma_0}{J(1-h)}}-\frac1J.$
We deduce that $F-G$ changes sign at least twice when $$\sigma_0<\frac{1-h}{4J},$$
and this ensures the existence of two steady states.
\end{proof}

Now we turn to the exponential stability of the unique steady state when $J$ is small.
It is a consequence of the following proposition, the proof of which is based on the Duhamel formula~\eqref{eq:Duhamel} combined with the exponential contraction of linear semigroup $(M_t)_{t\geq0}.$

\begin{proposition}
Let $(\mu_t)_{t\geq0}$ be a measure solution to Equation~\eqref{eq:LIF}
and let $\bar\mu$ be a steady state.
Then for all $t\geq0$ we have
\[\left\|\mu_t-\bar\mu\right\|_{\mathrm{TV}}\leq \frac{\e^{\omega t}}{1-c}\left\|\mu_0-\bar\mu\right\|_{\mathrm{TV}}\]
where the constants are given by
 $$c=\frac{\sigma_0}{2}\big(\frac h4\big)^{\sigma_0}, \quad \omega=\frac{2\sigma_0 J}{(1-c)(1-J)^2}+\frac{\log(1-c)}{\log\frac4h}.$$
\end{proposition}

\begin{proof}
First for $\mu\in\M([0,1])$ we define the measure $\mu\B$ by
\[\forall f\in C([0,1]),\qquad(\mu\B) f:=\mu(\B f).\]
The conservation property $\B\1=0$ ensures that for any $\mu$ element of $\M([0,1])$ we have $(\mu\B)([0,1])=0.$
This allows us to deduce from Proposition~\ref{prop:expo_contraction}, using also that $\B$ is bounded by $2,$ that for all $\mu\in\M([0,1])$ and all $t\geq0$
\[\left\|\mu\B M_t\right\|_{\mathrm{TV}}\leq \e^{-a(t-t_0)}\left\|\mu\B\right\|_{\mathrm{TV}}\leq2\,\e^{-a(t-t_0)}\left\|\mu\right\|_{\mathrm{TV}},\]
where $$t_0=\log\frac4h, \quad a=\frac{-\log(1-c)}{t_0}.$$
Using this inequality in the Duhamel formula~\eqref{eq:Duhamel} with $\sigma=\bar\sigma$ we get
\begin{align*}
\left\|\mu_t-\bar\mu\right\|_{\mathrm{TV}}&\leq\left\|(\mu_0-\bar\mu)M_t\right\|_{\mathrm{TV}}+\sigma_0J\int_0^t\bigg|\frac{(\mu_s-\bar\mu)([1-h,1])}{(1-J\mu_s([1-h,1]))(1-J\bar\mu([1-h,1]))}\bigg|\,\left\|\mu_s\B M_{t-s}\right\|_{\mathrm{TV}}\,ds\\
&\leq\left\|\mu_0-\bar\mu\right\|_{\mathrm{TV}}\e^{-a(t-t_0)} +\frac{2\sigma_0J}{(1-J)^2}\int_0^t\left\|\mu_s-\bar\mu\right\|_{\mathrm{TV}}\,\e^{-a(t-s-t_0)}ds.
\end{align*}
Denoting $$\theta(t)=\left\|\mu_t-\bar\mu\right\|_{\mathrm{TV}}\,\e^{at}$$ this also reads
\[\theta(t)\leq\left\|\mu_0-\bar\mu\right\|_{\mathrm{TV}}\,\e^{a t_0}+\frac{2\sigma_0J\e^{a t_0}}{(1-J)^2}\int_0^t\theta(s)\,ds=\frac{1}{1-c}\left\|\mu_0-\bar\mu\right\|_{\mathrm{TV}}+\frac{2\sigma_0J}{(1-c)(1-J)^2}\int_0^t\theta(s)\,ds\]
and the Gr\"onwall's lemma ensures that
\[\theta(t)\leq\frac{\e^{\frac{2\sigma_0J}{(1-c)(1-J)^2}t}}{1-c}\left\|\mu_0-\bar\mu\right\|_{\mathrm{TV}}.\]
\end{proof}

\begin{corollary}
If the following condition holds $$J<(5-2\sqrt{6})\big(\frac{h}{4}\big)^{\sigma_0+1}$$ then the steady state is unique and globally exponentially stable.
\end{corollary}

\begin{proof}
We only have to check that if $$J<(5-2\sqrt{6})\big(\frac{h}{4}\big)^{\sigma_0+1}$$ then $$\frac{2\sigma_0 J}{(1-c)(1-J)^2}<\frac{-\log(1-c)}{\log\frac4h}.$$\\
Using that $$\log(x)\leq x-1$$ we get
\[\frac{-\log(1-c)}{2\sigma_0\log\frac4h}\geq\frac{c}{2\sigma_0\log\frac4h}=\frac{1}{4\log\frac4h}\Big(\frac{h}{4}\Big)^{\sigma_0}\geq\frac{1}{4}\Big(\frac{h}{4}\Big)^{\sigma_0+1}.\]
Since 
$$J_0:=5-2\sqrt{6}\in(0,1) \Rightarrow \frac{J_0}{(1-J_0)^2}=\frac18,$$ so that $$J<J_0\big(\frac{h}{4}\big)^{\sigma_0+1}, \Rightarrow  \frac{J}{(1-J)^2}<\frac{1}{8}\Big(\frac{h}{4}\Big)^{\sigma_0+1}.$$
The conclusion follows from the bound
\[c\leq\frac{1}{2\e\log(\frac4h)}\leq\frac{1}{4\e\log2}<\frac{1}{2}.\]

\end{proof}

\section{Conclusion}\label{sec:conclusion}
The mean-field model considered along this paper is
a standard equation capturing the spiking population rate of a local neural circuit \cite{deco2008}. While not
specifically a model of any particular brain region, it
 describes a population of self recurrent excitatory LIF neurons receiving stochastic Poisson spike trains. Although the mean-field equation (\ref{eq:LIF}) is widespread among physicists, it has received only little attention by mathematicians, and there is nowadays, no identified mathematical framework to study its solution properties, see \cite{DH, DH02} for a first step in that direction. 
It has thus become necessary to investigate systematically the conditions under which the solution to the mean-field equation exists and to understand its stability properties.

In the mean-field limit, the level of recurrent excitation is control by a parameter $J$ which reflects the average number of connexion per cell. Interestingly, this parameter plays a critical part in the emergence of a finite time blow-up of the solution \cite{DH02}. This effect was first noticed in \cite{deville01} for the perfect integrate-and-fire, observed numerically with leaky integrate-and-fire neurons \cite{cascade01,cascade02}, and soon theoretically explained in \cite{DH02} using similar ideas to  \cite{carrillo01}. There is extensive numerical evidence that the blow-up of the mean-field equation is nothing but the emergence of synchrony patterns of firing across neurons. 

An important result that has been proved in \cite{deville01} is the existence and stability of a unique stationary state
for a moderate coupling scenario, i.e. for moderate values of $J$, the connectivity parameter. When the average number of connexions is not too big ($J<1$), the asynchronous state of a network of perfect integrate-and-fire neurons is stable.  
Our paper extends the stability property to networks with cells having a leaky membrane potential. Unfortunately, if we have been able to extend the existence of a steady state, the uniqueness and stability only hold for weak coupling ($J\ll1$).

Mean-field equations have gain intensive visibility over the past decades, however, most of the work has been done with the diffusion approximation equation. Assuming $h$ small enough, formal computations give:
\begin{equation*}
p(t,v)  - p(t,v-h) = h\dfrac{\partial}{\partial v} p(t,v)  -  \dfrac{h^2}{2} \dfrac{\partial^2}{\partial v^2}p(t,v) +o(h^2).
\end{equation*}
Plugging this second order approximation into the mean-field equation~\eqref{eq:LIF} as in~\cite{sirovich} leads to the diffusive PDE presented in \cite{B02,BH01} and studied mathematically in a sequel of papers \cite{carrillo01, perthame04, caceres2016,caceres2017,carrillo13,CarrilloPerthameSalortSmets}. Although the diffusion equation is more common in the literature - several textbooks dedicate a chapter to it \cite{Bres02,Ermentrout,gerstner,gerstner2014neuronal} -  and has the advantage to offer a clear expression of the steady state, recent modeling discussions suggested that it is not an appropriate description for most neural networks \cite{Iyer2013}. In any case, it seems crucial to us to relate our theoretical findings to the mathematical results established for the diffusion approximation. 

We first note that we get the same type of results for the stability of the steady state equation in the weakly coupled case ($J\ll1$). A difference should nonetheless be noted, for the diffusion equation, the exponential stability is only local  \cite{CarrilloPerthameSalortSmets}, while it is global in our case. Furthermore, with the diffusion equation, the global stability can not arise since it may blow up for a certain class of initial condition \cite{carrillo01}.
On the other hand, similar open issues hold for moderate coupling, where no precise conclusion can be formulated.
For the two models, depending on connectivity regimes, there can exist no steady state, one steady state, or at least two steady states.t
Numerically, both in the diffusion or non-diffusion scenario, the same steady state is always observed, suggesting that there is only one stable fixed point. 
Note that for strong coupling, for both mean-field equations, the steady-state does not exist, and obviously its stability property is not an issue  \cite{carrillo01,DH02}.

Probably, the most straightforward discussion that we should be having is about the stability and uniqueness of the steady state for moderate coupling. While the existence of a unique stable steady state has been addressed for an excitatory network of non leaky cells \cite{deville01}, it is still an open issue for the LIF. Another important discussion should address the discontinuous mechanism proposed by  \cite{deville01} to restart the flow of the solution after the blow-up. While there is no intuitive difficulties in proposing a similar discontinuous mapping for the PDE considered along this paper,  defining a solution at the blow-up time and extending it beyond the blow-up is not a trivial task and it will be the subject of a new research.

Note finally that in the present paper we considered excitatory networks.
A natural extension would be to consider networks made up of both excitatory and inhibitory neurons, as in~\cite{DH}.
In that case there are positive and negative jumps in the LIF model, and the presence of that latter leads to a mean-field equation set on $(-\infty,1)$ instead of $(0,1)$.
This prevents Doeblin's condition to be satisfied and one should instead use Harris's theorem, which extends Doeblin's ideas to the unbounded state space setting, see for instance~\cite{HairerMattingly,MeynTweedie}, and also~\cite{Bansaye2019,Cloez} for recent extensions to non-conservative semigroups.

\

\begin{appendices}

\section{Well-posedness in the linear case}\label{app:linear}

In this appendix we prove that the semigroup $(M_t)_{t\geq0}$ built in Section~\ref{sec:linear_case} provides the solutions to Equation~\eqref{eq:linear}.

\begin{proposition}\label{prop:linear_wellposed}
For every initial measure $\mu_0,$ the family $(\mu_0 M_t)_{t\geq0}$ is the unique solution to Equation~\eqref{eq:linear}.
\end{proposition}

The discontinuity of the indicator functions which appear in the operator $\B$ is an obstacle for proving directly Proposition~\ref{prop:linear_wellposed}.
To work around this difficulty, we use a regularization (see~\cite{EversHilleMuntean15} for a similar approach).
We approximate the indicator function $\1_{[1-h,1]}$ by
\[\chi_n(v):=\left\{\begin{array}{ll}
0&\text{if}\ v\leq 1-h-\frac hn,
\vspace{2mm}\\
1+\dfrac nh(v-1+h)\quad&\text{if}\ 1-h-\frac hn\leq v\leq 1-h,
\vspace{2mm}\\
1&\text{if}\ v\geq1-h,
\end{array}\right.\]
where $n$ belongs to $\N^*.$
The family $(\chi_n)_{n\geq1}$ is a decreasing sequence of continuous functions which converges pointwise to $\1_{[1-h,1]}.$
We define the associated regularized operators
\[\B^n f(v):=f(v+h)(1-\chi_n(v))+f(v_r)\chi_n(v)-f(v)\qquad\text{and}\qquad\A^n f(v):=-vf'(v)+\sigma_0\,\B^nf(v).\]
As for $\A$ we have the conservation property for $\A^n.$
But contrary to $\A,$ for $f$ element of $C^1([0,1])$ we have $\A^nf$ belongs to $ C([0,1]),$
and this allows us to build a measure solution to the regularized equation by duality.

\medskip

Consider the regularized dual equation
\begin{equation}\label{eq:dual_reg}
\partial_tf(t,v)+v\partial_vf(t,v)+\sigma_0f(t,v)=\sigma_0\big[f(t,v+h)(1-\chi_n(v))+f(t,v_r)\chi_n(v)\big],
\end{equation}
with the initial condition $f(0,\cdot)=f_0.$
As for the non-regularized case, this equation is well-posed on the space of continuous functions.
But it is also well-posed in the space of continuously differentiable functions.

\begin{lemma}\label{lm:lin_reg}
For $f_0$ an element taken from $C([0,1]),$ there exists a unique $f$ belonging to $ C(\R_+\times[0,1])$ which satisfies
\[f(t,v)=f_0(v\e^{-t})\e^{-\sigma_0t}+\sigma_0\int_0^t\e^{-\sigma_0\tau}\big[f(t-\tau,\e^{-\tau}v+h)(1-\chi_n(\e^{-\tau}v))+f(t-\tau,v_r)\chi_n(e^{-\tau}v)\big]d\tau.\]
Additionally
\begin{itemize}
\item if $f_0=\1$ then $f=\1,$
\item if $f_0\geq0$ then $f\geq0,$
\item if $f_0\in C^1([0,1])$ then $f\in C^1(\R_+\times[0,1])$ and $f$ satisfies~\eqref{eq:dual_reg}.
\end{itemize}
\end{lemma}

\begin{proof}
For the existence and uniqueness of a solution as well as the first two points we proceed as for Lemma~\ref {lm:fixedpoint} by applying the Banach fixed point theorem to the mapping
\[\Gamma f(t,v):=f_0(v\e^{-t})\,\e^{-\sigma_0t}+\sigma_0\int_0^t\e^{-\sigma_0\tau}\big[f(t-\tau,v\e^{-\tau}+h)(1-\chi_n(v\e^{-\tau}))+f(t-\tau,v_r)\chi_n(v\e^{-\tau})\big]\,d\tau.\]

It remains to check that when $f_0$ is of class $C^1$ then the same holds for $f.$
To do so we prove that when $f_0$ is an element of $ C^1([0,1])$ the mapping $\Gamma$ is a contraction in the Banach space $C^1([0,T]\times [0,1])$ endowed with the norm
$$\|f\|_{C^1}:=\|f\|_\infty+\|\partial_tf\|_\infty+\|\partial_vf\|_\infty$$ when $T$ is small enough.
We have
\begin{equation}\label{eq:d_tGamma}
\partial_t\Gamma f(t,v)=\A^nf_0(v\e^{-t})\,\e^{-\sigma_0t}+\sigma_0\int_0^t\e^{-\sigma_0\tau}\big[\partial_tf(t-\tau,v\e^{-\tau}+h)(1-\chi_n(v\e^{-\tau}))+\partial_tf(t-\tau,v_r)\chi_n(v\e^{-\tau})\big]\,d\tau.
\end{equation}
and
\begin{align*}
\partial_v\Gamma f(t,v)=&f'_0(v\e^{-t})\,\e^{-(1+\sigma_0)t}+\sigma_0\int_0^t\e^{-(1+\sigma_0)\tau}\partial_vf(t-\tau,v\e^{-\tau}+h)(1-\chi_n(v\e^{-\tau}))\,d\tau\\
&\qquad+\sigma_0\int_0^t\e^{-(1+\sigma_0)\tau}\frac nh\1_{[1-h-\frac hn,1-h]}(v\e^{-\tau})\big[f(t-\tau,v_r)-f(t-\tau,v\e^{-\tau}+h)\big]\,d\tau
\end{align*}
so when $f_0=0$ we have
\[\|\Gamma f\|_{C^1}\leq(1-\e^{-\sigma_0T})\|f\|_{C^1}+2\sigma_0\frac{n}{h}\log\Big(1+\frac{h}{n(1-h-\frac hn)}\Big)T\|f\|_\infty\leq\frac{1-2h+2\sigma_0}{1-2h}T\|f\|_{C^1}\]
and $\Gamma$ is a contraction in $C^1([0,T]\times [0,1])$ when $$T<\frac{1-2h}{1-2h+2\sigma_0}.$$
This ensures that the unique fixed point $f$ of $\Gamma$ belongs to $C^1([0,T]\times [0,1]).$
To check that $f$ satisfies~\eqref{eq:dual_reg} we can differentiate the alternative formulation of $\Gamma f$
\[\Gamma f(t,v)=f_0(v\e^{-t})\,\e^{-\sigma_0t}+\sigma_0\int_0^t\e^{-\sigma_0(t-\tau)}\big[f(\tau,v\e^{-(t-\tau)}+h)(1-\chi_n(v\e^{-(t-\tau)}))+f(\tau,v_r)\chi_n(v\e^{-(t-\tau)})\big]\,d\tau\]
with respect to $t$ and we get
\begin{align*}
\partial_t\Gamma f(t,v)=&-vf'_0(v\e^{-t})\,\e^{-(1+\sigma_0)t}-\sigma_0f_0(v\e^{-t})\,\e^{-\sigma_0t}+\sigma_0\big[f(t,v+h)(1-\chi_n(v))+f(t,v_r)\chi_n(v)\big]\\
&\quad-\sigma_0^2\int_0^t\e^{-\sigma_0(t-\tau)}\big[f(\tau,v\e^{-(t-\tau)}+h)(1-\chi_n(v\e^{-(t-\tau)}))+f(\tau,v_r)\chi_n(v\e^{-(t-\tau)})\big]\,d\tau\\
&\qquad-\sigma_0\int_0^tv\e^{-(1+\sigma_0)\tau}\partial_vf(t-\tau,v\e^{-\tau}+h)(1-\chi_n(v\e^{-\tau}))\,d\tau\\
&\quad\qquad-\sigma_0\int_0^tv\e^{-(1+\sigma_0)\tau}\frac nh\1_{[1-h-\frac hn,1-h]}(v\e^{-\tau})\big[f(t-\tau,v_r)-f(t-\tau,v\e^{-\tau}+h)\big]\,d\tau.
\end{align*}
So we have
\[\partial_t\Gamma f(t,v)+v\partial_v\Gamma f(t,v)+\sigma_0\Gamma f(t,v)=\sigma_0\big[f(t,v+h)(1-\chi_n(v))+f(t,v_r)\chi_n(v)\big]\]
and the fixed point satisfies~\eqref{eq:dual_reg}.

\end{proof}

With this result we define a conservative and positive contraction semigroup $(M^n_t)_{t\geq0}$ on $C([0,1])$ by setting $$M^n_tf_0=f(t,\cdot).$$
This regularized semigroup enjoys more properties than the non-regularized one.

\begin{lemma}
The semigroup $(M^n_t)_{t\geq0}$ is strongly continuous, meaning that for all $f$ taken from $C([0,1])$
\[\|M^n_tf-f\|_\infty\xrightarrow[t\to0]{}0,\]
Additionally for all $f$ belonging to $ C^1([0,1])$ we have
\begin{equation}\label{eq:d_tM^n}
\partial_tM^n_tf=\A^nM^n_tf=M^n_t\A^nf
\end{equation}
and
\[\bigg\|\frac1t\big(M^n_tf-f\big)-\A^nf\bigg\|_\infty\xrightarrow[t\to0]{}0.\]
\end{lemma}

\begin{proof}
The strong continuity follows from the fact that a continuous function on a compact set is uniformly continuous.

The first equality in~\eqref{eq:d_tM^n} is an immediate consequence of Lemma~\ref{lm:lin_reg}.
For the second equality we deduce from~\eqref{eq:d_tGamma} that $\partial_t M^n_tf$ is the unique fixed point of $\Gamma$ associated to $f_0=\A^nf,$ so
$$\partial_t M^n_tf=M^n_t\A^nf.$$

For the last point we use the strong continuity to write for $f$ an element of $ C^1([0,1])$
\[\bigg\|\frac1t\big(M^n_tf-f\big)-\A^nf\bigg\|_\infty\leq\frac1t\int_0^t\|M^n_s\A^nf-\A^nf\|_\infty\,ds\xrightarrow[t\to0]{}0,\]
since $\A^nf$ belongs to $ C([0,1]).$
\end{proof}

Now we can define by duality a semigroup on $$\M([0,1])=C([0,1])'.$$
For $\mu$ an element of $\M([0,1])$ and $t\geq0$ we define $\mu M^n_t$ by
\[\forall f\in C([0,1]),\qquad (\mu M^n_t)f=\mu(M^n_tf).\]
The family $(M^n_t)_{t\geq0}$ is then a positive and conservative contraction semigroup on $\M([0,1]),$ endowed with the total variation norm.
Additionally for all $\mu_0$ belonging to $\M([0,1])$ the family $(\mu_0 M^n_t)_{t\geq0}$ is a measure solution to the regularized leaky integrate-and-fire equation.

\begin{lemma}
For all $\mu$ taken from $\M([0,1])$ the application $$t\mapsto\mu M^n_t$$ is weak*-continuous, and for all $f$ belonging to $ C^1([0,1])$ and $t\geq0$
\begin{equation}\label{eq:weak_reg}
\mu M^n_t f = \mu f+\int_0^t\mu M^n_s\A^nf\,ds.
\end{equation}
\end{lemma}

\begin{proof}
The continuity of $$t\mapsto M^n_tf(v)$$ for all $f$ belonging to $C([0,1])$ and $v$taken from $[0,1]$ and the dominated convergence theorem ensure the weak*-continuity of $$t\mapsto\mu M^n_t.$$

For the second part of the lemma it suffices to integrate the identity $$\partial_sM^n_sf=M^n_s\A^nf$$ in time on $[0,t]$ and then in space on $[0,1]$ against the measure $\mu.$
The conclusion follows from the Fubini's theorem.
\end{proof}

It remains to pass to the limit when $n$ goes to infinity to get that the family $(\mu_0 M_t)_{t\geq0}$ is a measure solution to Equation~\eqref{eq:linear}.

\begin{lemma}\label{lm:pass_limit_lin}
For all $T>0$ we have
\[\sup_{0\leq t\leq T}\sup_{\|f\|_\infty\leq1}\|M^n_tf-M_tf\|_\infty\xrightarrow[n\to\infty]{}0.\]
\end{lemma}

\begin{proof}
From the definitions of $(M_t)_{t\geq0}$ and $(M^n_t)_{t\geq0}$ we get that for all $f$ element of $ C([0,1])$ such that $\|f\|_\infty\leq1$
\begin{align*}
\|M_tf-M^n_tf\|_\infty&\leq\sigma_0\int_0^t\|M_{t-\tau}f-M^n_{t-\tau}f\|_\infty d\tau+2\sigma_0\sup_{0\leq v\leq1}\int_0^t\1_{[1-h-\frac hn,1-h]}(\e^{-\tau}v)\,d\tau\\
&\leq\sigma_0\int_0^t\|M_{t-\tau}f-M^n_{t-\tau}f\|_\infty d\tau+2\sigma_0\log\bigg(1+\frac{h}{n(1-2h)}\bigg)
\end{align*}
and we conclude by the Gr\"onwall's lemma that
\[\|M_tf-M^n_tf\|_\infty\leq2\sigma_0\log\bigg(1+\frac{h}{n(1-2h)}\bigg)\,\e^{\sigma_0t}.\]
\end{proof}

\begin{proof}[Proof of Proposition~\ref{prop:linear_wellposed}]
Let $\mu_0$ be an element of $\M([0,1]).$
From Lemma~\ref{lm:pass_limit_lin} we deduce that $$\mu_0 M^n_t\to\mu_0 M_t$$ in the TV-norm when $n$ goes to infinity.
This allows us to pass to the limit in~\eqref{eq:weak_reg} by dominated convergence,
since for all $f$ belongs to $ C^1([0,1])$ we have $$\A^nf\to\A f$$ pointwise and $$\|\A^nf\|_\infty\leq \|f'\|_\infty+2\,\|f\|_\infty.$$

\medskip

The weak*-continuity of $$t\mapsto\mu_0 M_t$$ follows from the weak*-continuity of $$t\mapsto\mu_0 M^n_t,$$ using again Lemma~\ref{lm:pass_limit_lin}.

\medskip

For the uniqueness we use that if $(\mu_t)_{t\geq0}$ is a solution to Equation~\eqref{eq:linear} then for all $n\in\N^*,$ all $t>0,$ and all $f$ taken in $C^1([0,1])$ we have
\begin{align}\label{eq:Duhamel_lin}
\frac{d}{ds}\bigg(\int_0^s\mu_\tau M^n_{t-s}f\,d\tau\bigg)&=\mu_s M^n_{t-s}f-\int_0^s\mu_\tau\,\A^nM^n_{t-s}f\,d\tau\nonumber\\
&=\mu_0 M^n_{t-s}f+\int_0^s\mu_\tau(\A-\A^n)M^n_{t-s}f\,d\tau\nonumber\\
&=\mu_0 M^n_{t-s}f+\int_0^s\mu_\tau(\B-\B^n)M^n_{t-s}f\,d\tau.
\end{align}
For proving the validity of the differentiation we write for all $h>0$
\begin{align*}
\frac1h\bigg[\int_0^{s+h}\mu_\tau&M^n_{t-s-h}f\,d\tau-\int_0^s\mu_\tau M^n_{t-s}f\,d\tau\bigg]=\\
&\frac1h\int_s^{s+h}\mu_\tau M^n_{t-s}f\,d\tau+\int_s^{s+h}\mu_\tau\frac{M^n_{t-s-h}f-M^n_{t-s}f}{h}\,d\tau+\int_0^s\mu_\tau\frac{M^n_{t-s-h}f-M^n_{t-s}f}{h}\,d\tau.
\end{align*}
The convergence of the first term above is a consequence of the weak*-continuity of $\tau\mapsto\mu_\tau$
\[\frac1h\int_s^{s+h}\mu_\tau M^n_{t-s}f\,d\tau\xrightarrow[h\to0]{}\mu_sM^n_{t-s}f.\]
For the second term we use that $\tau\mapsto\mu_\tau$ is locally bounded for the TV-norm due to the uniform boundedness principle, because it is weak*-continuous.
Using~\eqref{eq:d_tM^n} we deduce
\[\bigg|\int_s^{s+h}\mu_\tau\frac{M^n_{t-s-h}f-M^n_{t-s}f}{h}\,d\tau\bigg|\leq h\sup_{s\leq \tau\leq s+h}\left\|\mu_\tau\right\|_{\mathrm{TV}}\,\|\A^nf\|_\infty\xrightarrow[h\to0]{}0.\]
For the last term we also use~\eqref{eq:d_tM^n} to get by dominated convergence
\[\int_0^s\mu_\tau\frac{M^n_{t-s-h}f-M^n_{t-s}f}{h}\,d\tau\xrightarrow[h\to0]{}-\int_0^s\mu_\tau\,\A^nM^n_{t-s}f\,d\tau.\]
Now that~\eqref{eq:Duhamel_lin} is proved, we integrate on $[0,t]$ to obtain
\[\int_0^t\mu_\tau f\,d\tau=\int_0^t\mu_0 M^n_{t-s}f\,ds+\int_0^t\int_0^s\mu_\tau(\B-\B^n)M^n_{t-s}f\,d\tau\,ds\]
and by dominated convergence, when $n$ goes to infinity.
\[\int_0^t\mu_\tau f\,d\tau=\int_0^t\mu_0 M_{s}f\,ds.\]
Differentiating this identity with respect to $t$ we get that $$\mu_tf=\mu_0 M_{t}f$$ and then $$\mu_t=\mu_0 M_t$$ because $C^1([0,1])$ is a dense subspace of $C([0,1]).$
\end{proof}

\end{appendices}

\

\noindent{\bf Acknowledgments.}
P.G. has been supported by the ANR project KIBORD, ANR-13-BS01-0004, funded by the French Ministry of Research.


\bibliographystyle{abbrv}
\bibliography{steady-states_LIF}


\end{document}